\documentclass[reqno]{amsart}

\usepackage{mathrsfs}
\usepackage{amsfonts}
\usepackage{amsmath}
\usepackage{amssymb}
\usepackage{centernot}
\usepackage[table,xcdraw]{xcolor}
\usepackage{pifont}
\usepackage{mathtools}
\usepackage{float}
\usepackage{xcolor}
\usepackage{mathdots}
\usepackage{enumitem}
\usepackage{array}
\usepackage{subcaption}
\usepackage{mathdots}
\usepackage{tikz}
\usepackage{tikz-cd}
\usetikzlibrary{arrows,automata}
\usetikzlibrary{decorations.markings}
\usetikzlibrary{automata, positioning, arrows}
\usetikzlibrary{math}
\usepackage{hyperref}
\usepackage{centernot}
\usepackage{marginnote}

\newcommand{\pres}[3]{\textnormal{#1} \langle #2 \mid #3 \rangle}

\newcommand{\Z}{\mathbb{Z}}

\newcommand{\leqfg}{\leq_{\mathrm{f.g.}}}
\newcommand{\leqfi}{\leq_{\mathrm{f.i.}}}

\newcommand{\cc}{\mathbf{C}}

\newcommand{\CF}{\mathbf{CF}}
\newcommand{\pCF}{\mathbf{pCF}}
\newcommand{\CS}{\mathbf{CS}}

\newcommand{\REG}{\mathbf{Reg}}
\newcommand{\MCF}{\mathbf{MCF}}

\newcommand{\REC}{\mathbf{REC}}
\newcommand{\RE}{\mathbf{RE}}
\newcommand{\Fin}{\text{Fin}}
\newcommand{\OC}{\mathbf{1C}}
\newcommand{\ETOL}{\mathbf{ET0L}}

\newcommand{\cR}{\mathcal{R}}

\DeclareMathOperator{\WP}{WP}
\DeclareMathOperator{\CP}{CP}
\DeclareMathOperator{\Aut}{Aut}

\newcommand{\AFL}{\mathrm{AFL}}

\newtheorem{theorem}{Theorem} 
\numberwithin{theorem}{section}
\newtheorem*{theorem*}{Theorem} 
\newtheorem{lemma}[theorem]{Lemma}     
\newtheorem{corollary}[theorem]{Corollary}

\newtheorem{proposition}[theorem]{Proposition}

\newtheorem*{mainlemma*}{Main Lemma}
\newtheorem{conjecture}{Conjecture}
\newtheorem*{conjecture*}{Conjecture}

\theoremstyle{definition}
\newtheorem{definition}{Definition}
\newtheorem{question}{Question}
\newtheorem*{question*}{Question}
\newtheorem{example}{Example}
\numberwithin{example}{section}
\newtheorem{remark}{Remark}

\begin{document}

\title[On linguistic subsets of groups]{On linguistic subsets of groups and monoids} 

\author{André Carvalho}
\address{Center of Mathematics of the University of Porto, Department of Mathematics, Faculty of Sciences of the University of Porto\\ Rua do Campo Alegre, s/n, 4169-007 Porto, Portugal}
\email{andrecruzcarvalho@gmail.com}

\author{Carl-Fredrik Nyberg-Brodda}
\address{June E.\ Huh Center for Mathematical Challenges, Korea Institute for Advanced Study (KIAS), Seoul 02455, Republic of Korea}
\email{cfnb@kias.re.kr}

\thanks{The first author was supported by CMUP, which is financed by national funds through FCT -- Fundação para a Ciência e a Tecnologia, I.P., under the projects with reference UIDB/00144/2020 and UIDP/00144/2020. The second author is supported by Mid-Career Researcher Program (RS-2023-00278510) through the National Research Foundation funded by the government of Korea, and by the KIAS Individual Grant (MG094701) at Korea Institute for Advanced Study.}

\subjclass[2020]{}

\date{\today}


\keywords{}

\begin{abstract}
We study subsets of groups and monoids defined by language-theoretic means, generalizing the classical approach to the word problem. We expand on results by Herbst from 1991 to a more general setting, and for a class of languages $\cc$ we define the classes of $\cc^\forall$-flat and $\cc^\exists$-flat groups. We prove several closure results for these classes of groups, prove a connection with the word problem, and characterize $\cc^\forall$-flat groups for several classes of languages. In general, we prove that the class of $\cc^\forall$-flat groups is a strict subclass of the class of groups with word problem in $\cc$, including for the class $\REC$ of recursive languages, for which $\cc^\forall$-flatness for a group resp.\ monoid is proved to be equivalent to the decidability of the subgroup membership problem resp.\ the submonoid membership problem. We provide a number of examples, including the Tarski monsters of Ol'shanskii, showing the difficulty of characterizing $\cc^\exists$-flat groups. As an application of our general methods, we also prove in passing that if $\cc$ is a full semi-$\AFL$, then the class of epi-$\cc$ groups is closed under taking finite index subgroups. This answers a question recently posed by Al Kohli, Bleak \& Elliott. 
\end{abstract}

\maketitle

\noindent Groups have been studied by language-theoretic means for the past half century. Such methods were first considered by Anisimov \cite{Anisimov1971} in 1971, who characterized finite groups by means of regular languages, and initiated the investigation of groups with context-free \textit{word problem}. Here, the word problem refers to the language of all words, over some finite generating set, which represent the identity element of the group. That line of investigation received its capstone in 1983 by Muller \& Schupp \cite{Muller1983}, who proved that a finitely generated group has context-free word problem if and only if it is virtually free, i.e.\ has a free subgroup of finite index (this result relies on a subsequent theorem on the \textit{accessibility} of finitely presented groups by Dunwoody \cite{Dunwoody1985}). Since then, language-theoretic methods and vocabulary have permeated many parts of group theory, from the theory of automatic groups \cite{Wordprocessing}, to indexed combings of $3$-manifold groups \cite{Bridson1993,Bridson1996}, to regular languages of geodesics in hyperbolic groups. 

In this article, we will study subsets of groups and monoids defined by language-theoretic means. Specifically, let $M$ be a monoid (or group) generated by some finite set $A$, with associated surjective homomorphism $\pi \colon A^\ast \to M$. Let $\cc$ be a class of languages. Then certain subsets of $M$ can be distinguished using the class $\cc$ together with the homomorphism $\pi$. Specifically, we say that $X \subseteq M$ is a $\cc^\forall$\textit{-subset} if $\pi^{-1}(X) \in \cc$, and we say that $X$ is a $\cc^\exists$\textit{-subset} if there exists some $L \in \cc$ such that $\pi(L) = X$. One particularly simple example comes from the classical word problem: if $M$ is a group, then the word problem of $M$ lies in $\cc$ if and only if $\{ 1 \}$ is a $\cc^\forall$-subset. For ease of notation, we often discuss $\cc^\bullet$-subsets ambiguously, referring to results that hold true for both $\cc^\forall$-subsets and $\cc^\exists$-subsets verbatim. In general, $\cc^\bullet$-subsets have been studied in the literature before: for example, when $\cc = \REG$, the class of regular languages, the set of $\cc^\forall$-subsets resp.\ $\cc^\exists$-subsets are called the \textit{recognizable} resp.\ \textit{rational subsets} of $M$. However, in this article we take a foundational and particularly general approach, building on work by Herbst \cite{Herbst1991}. We will also define a notion of \textit{relative $\cc^\bullet$-subset} for finitely generated subgroups and submonoids.

The outline of the article is as follows. In \S\ref{Sec:1Prelim}, we present the necessary language-theoretic background, together with a brief overview of its connection with group-theoretic and monoid-theoretic results. In \S\ref{Sec:2ccsubsets}, we define the notion of \textit{$\cc^\bullet$-subsets} and \textit{relative $\cc^\bullet$-subsets}, and set up a general framework for studying such subsets. We characterize the relative $\cc^\bullet$-subsets of a finite index subgroup inside a group, and use this to give an answer to a recent question of Al Kohli, Bleak \& Elliott regarding \textit{epi-$\cc$ groups}. In \S\ref{Sec:3flatmonoidsgeneral}, we introduce the notion of \textit{$\cc^\bullet$-flat} monoids and groups, and study some general properties of these classes. In \S\ref{Sec:4flatgroupsforspecific}, we use these general properties to give a full characterization of $\cc^\forall$-flat groups for certain classes of languages, including the one-counter languages $\OC$, the context-free languages $\CF$, the recursively enumerable languages $\RE$, and the recursive languages $\REC$. For the last class, we prove the equivalence of $\cc^\forall$-flatness as a group resp.\ as a monoid with the decidability of the subgroup resp.\ submonoid membership problem. Using the Tarski monsters of Ol'shanskii, we prove that the class of $\RE^\exists$-flat groups is strictly larger than the class of $\RE^\forall$-flat groups. Finally, in \S\ref{Sec:5openproblems} we present a tentative link to the conjugacy problem, and list a number of open problems and questions.

\section*{Acknowledgements} 

The authors wish to thank Raad Al Kohli, Collin Bleak, and Luna Elliott for helpful discussions pertaining to their preprint \cite{Kohli2025}, and Corentin Bodart for many interesting and helpful discussions and pointers leading to a number of improvements, including a strengthening of Proposition~\ref{Prop:cc-flatness-preserved-by-finite-index}.

\section{Preliminaries}\label{Sec:1Prelim}

\noindent In this section, we will fix some notation and present the necessary background required for our investigations. None of the definitions or results presented herein are new. We assume the background is familiar with the rudiments of formal language theory and combinatorial group theory, particularly presentations of groups, and refer the reader to \cite{Hopcroft1979} for a general introduction to the former, and to \cite{MagnusKarrassSolitar, LyndonSchupp} for the latter. 

\subsection{Language-theoretic background}

Let $\mathfrak{A}$ be a countably infinite set of symbols; this set is considered as an ambient notational convenience, and we will assume it contains all symbols we might ever require of it. An \textit{alphabet} is any finite subset $A \subset \mathfrak{A}$. The free semigroup resp.\ the free monoid on $A$ is denoted $A^+$ resp.\ $A^\ast$, and the empty word is written $1 \in A^\ast$. A \textit{language} $L$ is a set such that there is some $A \subset \mathfrak{A}$ such that $L \subseteq A^\ast$. A \textit{class of languages} $\cc \subset 2^{\mathfrak{A}^\ast}$ is a non-empty set of non-empty languages. A \textit{monoid presentation} is a pair $(A, \mathcal{R})$ consisting of an alphabet $A$ and a set of pairs of words $\mathcal{R} \subseteq A^\ast \times A^\ast$. Letting $A^{-1}$ denote a set disjoint from and in bijective correspondence with $A$ via an involution $a \mapsto a^{-1}$, we can consider the set $\mathcal{T}_A$ of all ``trivial'' group relations $(aa^{-1}, 1)$ and $(a^{-1}a, 1)$ for all $a \in A$. The monoid presentation $(A \cup A^{-1}, \cR \cup \mathcal{T}_A)$ is then called a \textit{group presentation}. A monoid presentation resp.\ a group presentation $(A, \mathcal{R})$ is denoted $\pres{Mon}{A}{\mathcal{R}}$ resp.\ $\pres{Gp}{A}{\mathcal{R}}$. For notational ease, a relation $(u, v)$ will always be denoted $u = v$. Any monoid presentation $\pres{Mon}{A}{\cR}$ gives rise to a monoid as the quotient $A^\ast / \varrho_{\cR}$, where $\varrho_{\cR}$ is the least congruence on $A^\ast$ containing $\cR$. Any finitely generated monoid clearly admits a presentation. In the case of a group presentation, the corresponding quotient is clearly always a group. Throughout this article, if nothing else is specified, all groups and monoids will be assumed to be finitely generated. 

Because there are $2^\mathfrak{c} = \beth_2$ possible classes of languages, it is natural to consider only consider certain special classes. We assume the reader is familiar with the classes of regular, context-free, context-sensitive, recursive, and recursively enumerable languages; which we denote by $\REG, \CF, \CS, \REC$ and $\RE$, respectively. Details on all these languages can be found in \cite{Hopcroft1979}. The class of \textit{one-counter languages} is denoted $\OC$, see \cite{Herbst1991}. Generally speaking, if $\cc$ is a class of languages, then we say that $\cc$ is closed under intersection with regular languages if for all $L \in \cc$ and $R \in \REG$ we have $L \cap R \in \cc$. We say that $\cc$ is closed under homomorphisms if $\cc$ contains all homomorphic images of all languages in $\cc$ with respect to homomorphisms $\varphi \colon A^\ast \to B^\ast$, where $A, B$ are alphabets. If the same is true when taking pre-images of such homomorphisms, then we say that $\cc$ is closed under inverse homomorphism. We say that $\cc$ is a \textit{cone} if $\cc$ is closed under homomorphism, inverse homomorphism, and intersection with regular languages. If $\cc$ is closed under inverse homomorphism and intersection with regular languages, and additionally closed under homomorphisms whose image do not contain the empty word (\textit{non-erasing homomorphisms}), then $\cc$ is called a \textit{trio}. If $\cc$ is a cone and closed under union, then we say that $\cc$ is a \textit{full semi-$\AFL$} (where $\AFL$ abbreviates \textit{abstract family of languages}). Finally, if $\cc$ is a full semi-$\AFL$ closed under concatenation and the Kleene star $^\ast$, then we say that $\cc$ is a \textit{full $\AFL$.} We remark that any trio (and so, in particular, any full semi-AFL) contains the class of all regular languages.

\subsection{Groups and languages}

Monoid presentations, and especially group presentations, give rise to many natural languages, and studying languages defined in group-theoretic ways has a long history. In this section, we present some of the fundamental results in this line. Let $G$ be a finitely generated group, let $A$ be a finite (monoid) generating set for $G$, and let $\pi \colon A^\ast \to G$ be a surjective homomorphism. Then the set 
\begin{equation}\label{Eq:WP-definition}
\WP_A^G = \{ w \mid w \in A^\ast, \pi(w) = 1 \}
\end{equation}
of all words in $A$ representing the identity element of $G$ is called the \textit{word problem} of $G$. The study of this language led to the introduction of language-theoretic methods in group theory, with the following three results being central: 

\begin{theorem}\label{Thm:AnisimovMSHerbst}
Let $G = \pres{}{A}{R}$ be a finitely generated group. Then we have:
\begin{enumerate}
\item  $\WP_A^G \in \REG$ $\iff$ $G$ is finite (Anisimov, 1971 \cite{Anisimov1971}).
\item  $\WP_A^G \in \CF$ $\iff$ $G$ is virtually free (Muller \& Schupp, 1983 \cite{Muller1983}).
\item $\WP_A^G \in \OC$ $\iff$ $G$ is virtually cyclic (Herbst, 1991 \cite{Herbst1991}).
\end{enumerate}
\end{theorem}

We remark that the independence of finite generating set $A$ chosen for $G$ for the language-theoretic properties of $\WP_A^G$, at least for the three classes in Theorem~\ref{Thm:AnisimovMSHerbst}, is easily seen from the algebraic nature of the right-hand side of all biconditionals in the same theorem. More generally, it is not difficult to show that if $\cc$ is a cone, then if $A, B$ are two finite generating sets of a monoid $M$ we have that $\WP_A^M \in \cc$ if and only if $\WP_B^M \in \cc$. 

\begin{remark}\label{Rem:Monoids-WP}
All of the definitions in this present article carry over verbatim from groups to monoids. Some of our results also hold for monoids, and as such we will state most results in a general way. However, in the case of the word problem \eqref{Eq:WP-definition}, very little can be said in the case of monoids: indeed, for many monoids, this language carries essentially no algebraic information whatsoever. Alternative definitions of a monoid-theoretic word problem have been proposed, with one by Duncan \& Gilman \cite{Duncan2004} having seen a good deal of activity in the past few years \cite{Brough2019,Kostolanyi2025,NybergBrodda2022}. In particular, it can be shown that the set \eqref{Eq:WP-definition} carries much algebraic information when all defining relations of the monoid are of the form $r = 1$, i.e.\ when the monoid is \textit{special}, see \cite{NybergBrodda2022}. 
\end{remark}

The language of all words representing just the identity element is rather restrictive (especially for monoids, as indicated in Remark~\ref{Rem:Monoids-WP}), and it is thus of interest to place this language in a more general setting. Two early such generalizations come from \textit{rational} and \textit{recognizable} subsets. Let $M$ be a finitely generated group, $A$ be a finite set of semigroup generators and $\pi:A^*\to G$ be a surjective homomorphism. A subset $X\subseteq M$ is said to be \textit{rational} if there is some regular language $L\subseteq \tilde A^*$ such that $\pi(L)=K$. We say that $X$ is \textit{recognizable} if the full pre-image $\pi^{-1}(X)$ is a regular language. Clearly, every recognizable subset is rational. On the other hand, in the case of a group $G$, the subset $\{ 1 \} \subseteq G$ is clearly always rational, but it is recognizable if and only $G$ is finite, by Theorem~\ref{Thm:AnisimovMSHerbst}. An early success in the study of the rational and recognizable subsets appeared in 1975, and can be stated as follows:

\begin{theorem}[Anisimov \& Seifert \cite{Anisimov1975}]\label{Thm:Anisimov-Seifert}
Let $G$ be a finitely generated group, and let $H \leq G$. Then:
\begin{enumerate}
\item $H$ is a rational subset of $G$ if and only if $H$ is finitely generated.
\item $H$ is a recognizable subset of $G$ if and only if $[G : H] < \infty$. 
\end{enumerate}
\end{theorem}

In fact, there is an explicit description of the recognizable subsets of groups: namely, any recognizable subset is a finite union of cosets of a finite index subgroup. Next, Frougny, Sakarovitch \& Schupp \cite{Frougny1989} noted that the proof given by Anisimov \& Seifert in proving Theorem~\ref{Thm:Anisimov-Seifert} actually yields the following strong result: 

\begin{proposition}\label{Prop:All-REGe-flat-old}
Let $X$ be a rational subset of a group $G$, and let $H \leqfg G$ be a finitely generated subgroup of $G$. If $X$ is contained in $H$, then $X$ is a rational subset also as a subset of $H$. 
\end{proposition}

In this present article, we will refer to such a property as \textit{flatness}, and we will present it in more detail in \S\ref{Subsec:flatness}; in particular, we will rephrase Proposition~\ref{Prop:All-REGe-flat-old} in our new language as simply stating that \textit{all groups are $\REG^\exists$-flat} (see Proposition~\ref{Prop:AnisimovSeifert-rephrased-as-flatness}). Furthermore, in Carvalho \cite{Carvalho2023}, some results on extending Proposition~\ref{Prop:All-REGe-flat-old} to context-free languages are given; we will present these results shortly (Theorem~\ref{Thm:REG-forall-CF-forall-characterized}). The overall goal of the present article is to provide a unified framework for studying subsets of groups defined in the above language-theoretic manner, and to provide generalizations of results of the same type as Theorem~\ref{Thm:Anisimov-Seifert} and Proposition~\ref{Prop:All-REGe-flat-old}.

\section{General properties of $\cc$-subsets}\label{Sec:2ccsubsets}

\noindent In this section, we set up a framework for studying subsets of groups and monoids defined by language-theoretic means. We begin with a general definition of $\cc$-subsets, following Herbst \cite{Herbst1991}, and establish some general properties of such subsets. We will then introduce the key new idea of $\cc^\forall$\textit{-flat} and $\cc^\exists$\textit{-flat} groups and monoids (ambiguated as $\cc^\bullet$-flatness), which will be central throughout this article.

\subsection{$\cc^\bullet$-subsets}

Let $\cc$ be a class of languages. Let $M$ be a finitely generated monoid, with finite monoid generating set $A$ and associated surjective homomorphism $\pi \colon A^\ast \to G$. Let $S$ be a subset of $M$. Then we say that $S$ is a \textit{$\cc^\forall$-subset of $M$} if its full pre-image under $\pi$ is a language in $\cc$. We say that $S$ is a \textit{$\cc^\exists$-subset of $M$} if there is some language in $\cc$ whose image under $\pi$ is $S$. The set of all $\cc^\forall$-subsets of $M$ is denoted $\cc^\forall(M)$, and analogously the set of all $\cc^\exists$-subsets of $M$ is denoted $\cc^\exists(M)$. Symbolically, we have the following:
\begin{align}
S \in \cc^\forall(M) \quad &\iff \quad \exists L \in \cc, \forall s \in S, \forall x \in \pi^{-1}(s) : \left( \pi(L) = S \text{ and } x \in L \right), \\
S \in \cc^\exists(M) \quad &\iff \quad \exists L \in \cc, \forall s \in S, \exists x \in \pi^{-1}(s) : \left( \pi(L) = S \text{ and } x \in L \right).
\end{align}
This notation and terminology was introduced by Herbst \cite{Herbst1991}, in his study of subsets of context-free groups. We note that the only difference in the two definitions above is the single quantifier on the elements of the pre-image, which gives rise to our notation. Furthermore, we clearly have $\cc^\forall(M) \subseteq \cc^\exists(M) \subseteq 2^M$ for all classes $\cc$ and all monoids $M$. Throughout this article, many of our elementary, and some more complicated, results will be equally true for $\cc^\forall$-subsets as for $\cc^\exists$-subsets. We will thus occasionally ambiguously refer to $\cc^\bullet$-\textit{subsets} of $M$, and will use this notation e.g.\ in theorems that are true for the set of $\cc^\forall$-subsets \textit{and} for the set of $\cc^\exists$-subsets. 

Note that the notation $\cc^\bullet(M)$ is somewhat underspecified, in that it does not refer to the choice of finite generating set $A$, and one might wish to introduce notation akin to $\cc^\bullet_A(M)$ to remedy this. However,  it turns out that if $\cc$ is a class of languages closed under inverse homomorphisms, then the set of $\cc^\bullet$-subsets of $M$ does not depend on the generating set chosen. This is a direct consequence of the following somewhat technical, but otherwise simple, lemma.

\begin{lemma}[Herbst {\cite[Lemma~4.3]{Herbst1991}}]\label{Lem:Herbst-lemma-2.3}
Let $M$ and $M'$ be finitely generated monoids, generated by $A$ and $A'$, respectively, with associated surjections $\varphi \colon A^\ast \to M$ and $\varphi' \colon (A' )^\ast \to M'$. Let $\tau \colon M' \to M$ be any homomorphism, and let $X \subseteq M$ be any subset, setting $X' = \tau^{-1}(X) \subseteq M'$. Let $L_1 = (\varphi')^{-1}(T')$ and $L_2 = \varphi^{-1}(T)$. Then there exists a homomorphism $h \colon (A')^\ast \to A^\ast$ such that $L_1 = h^{-1}(L_2)$. 
\[
\begin{tikzcd}
L_1 \subseteq (A')^\ast \arrow[d, "\varphi'"'] \arrow[r, "h", dotted] & A^\ast \supseteq L_2 \arrow[d, "\varphi"] \\
X' \subseteq M' \arrow[r, "\tau"']                                    & M \supseteq X                            
\end{tikzcd}
\]
\end{lemma}

In all cases arising in practice, the language classes $\cc$ we will consider will always be closed under inverse homomorphism. However, we will always specify this assumption when we use it in our theorems; thus, in the interest of avoiding clumsy notation, in all cases that this assumption is not explicitly specified, our theorems will be shown to hold for \textit{all} choices of finite generating sets $A$ (in practice, this is never a complicating factor). We now give some fundamental examples of $\cc^\bullet$-subsets, connecting it to existing terminology in the literature.

\begin{example}\label{Examples}
For any finitely generated monoid $M$, we have: 
\begin{enumerate}
\item $\REG^\exists(M)$ is the set of \textit{rational} subsets of $M$. Such subsets generalize finitely generated submonoids and subgroups, and are important from the point of view of the rational subset \textit{membership problem}, see e.g.\ \cite{Lohrey2015, Bartholdi2021}. By Theorem~\ref{Thm:Anisimov-Seifert}, for any group $G$, if $H \leq G$ is a subgroup then $H \in \REG^\exists(G)$ if and only if $H$ is finitely generated. 
\item $\REG^\forall(M)$ is the set of \textit{recognizable} subsets of $M$.  By Theorem~\ref{Thm:Anisimov-Seifert}, if $G$ is a group and $H \leq G$ is a subgroup, then $H \in \REG^\forall(G)$ if and only if the index of $H$ is finite in $G$, i.e.\ $[G : H] < \infty$. 
\item If $F$ is a finitely generated free monoid, then Kleene's Theorem \cite{Hopcroft1979} states that $\REG^\forall(F) = \REG^\exists(F)$. This equality also holds for finite monoids and, more generally, for \textit{rational} monoids in the sense of Sakarovitch \cite{Sakarovitch1987}.
\item $\CF^\exists(M)$ is called the set of \textit{algebraic} subsets of $M$, and $\CF^\forall(M)$ is called the set of \textit{context-free} subsets of $M$. The main result of Herbst \cite{Herbst1991} states that if $G$ is a finitely generated group, then $\CF^\forall(G) = \CF^\exists(G)$ if and only if $G$ is virtually cyclic, in which case also $\CF^\forall(G) = \REG^\exists(G) = \OC^\forall(G)$. 
\item Let
  $\cc$ be a class of languages closed under inverse homomorphism and containing all singleton languages, and let $G$ a group generated by the finite set $A$. Then the word problem $\WP_A^G$, in the sense of \eqref{Eq:WP-definition}, by definition satisfies: 
\[
\WP_A^G \in \cc \quad \iff \quad \{ 1 \}  \in \cc^\forall(G).
\]
By contrast, since $\cc$ contains all singleton languages, it is easy to see that we \textit{always} have $\{ 1 \} \in \cc^\exists(G)$.
\end{enumerate}
\end{example}

Note that if $\cc_1, \cc_2$ are classes of languages such that $\cc_1 \subseteq \cc_2$, e.g.\ $\REG \subset \CF$, then for every finitely generated monoid $M$ we always have the following square:
\begin{equation}\label{Eq:squares-of-inclusions}
\begin{tikzcd}
\cc_1^\forall(M) \arrow[r, "\subseteq", hook] \arrow[d, "\subseteq"', hook] & \cc_2^\forall(M) \arrow[d, "\subseteq", hook] \\
\cc_1^\exists(M) \arrow[r, "\subseteq"', hook]                              & \cc_2^\exists(M)                             
\end{tikzcd}
 \qquad \text{e.g.} \qquad
\begin{tikzcd}
\REG^\forall(M) \arrow[r, "\subseteq", hook] \arrow[d, "\subseteq"', hook] & \CF^\forall(M) \arrow[d, "\subseteq", hook] \\
\REG^\exists(M) \arrow[r, "\subseteq"', hook]                              & \CF^\exists(M)                             
\end{tikzcd}
\end{equation}
In general, the relationship between the subsets $\cc_1^\exists(M)$ and $\cc_2^\forall(M)$ appears to be rather weak. We pose the following question as a natural starting point for further investigations along these lines: 

\begin{question}\label{Quest:can-square-collapse}
Do there exist classes of languages $\cc_1, \cc_2$ such that $\cc_1 \subsetneq \cc_2$ but $\cc_1^\exists(G) = \cc_2^\forall(G)$ for all finitely generated groups (or monoids) $G$?
\end{question}

Natural candidates for classes for which Question~\ref{Quest:can-square-collapse} may have a positive answer includes those which are difficult to separate using word problems of groups, e.g.\ $\CF \subset \ETOL$. We will not linger on this or related questions, and instead turn to proving elementary properties of $\cc^\bullet$-subsets of monoids.

\begin{lemma}\label{Lem:closed-under-intersection-2.1}
Let $\cc$ be a class of languages closed under intersection with regular languages. Let $M$ be a finitely generated monoid, and let $R \in \REG^\forall(M)$. Then:
\[
X \in \cc^\bullet(M)  \implies  (X \cap R) \in \cc^\bullet(M).
\]
\end{lemma}
\begin{proof}
We suppose $M$ is finitely generated by $A$, and let $\pi \colon A^\ast \to M$ be surjective. There are two claims to be proved: one for $\cc^\forall(M)$, and one for $\cc^\exists(M)$. We begin with $\cc^\forall(M)$. Let $X \in \cc^\forall(M)$, and for simplicity write $L := \pi^{-1}(X) \in \cc$. Then obviously
\begin{equation}\label{Eq:intersection-preimage}
\pi^{-1}(X \cap R) = \pi^{-1}(X) \cap \pi^{-1}(R) = L \cap \pi^{-1}(R)
\end{equation}
Since $R \in \REG^\forall(M)$, we have $\pi^{-1}(R) \in \REG$. Since $\cc$ is closed under intersection with regular languages, we thus have $\pi^{-1}(X \cap R) \in \cc$, so $(X \cap R) \in \cc^\forall(M)$. 

Next, consider the case of $X \in \cc^\exists(M)$. The above argument does not quite work in this case, since \eqref{Eq:intersection-preimage} cannot be used. However, since $R \in \cc^\forall(M)$, the argument can easily be modified to work. Let $L \in \cc$ be such that $\pi(L) = X$. Then $L \cap \pi^{-1}(R)$ has a representative for every element of $X \cap R$: indeed, if $x \in X \cap R$, then there is some $l \in L$ with $\pi(l) = x$, but since also $x \in R$ we have $l \in \pi^{-1}(x) \subseteq \pi^{-1}(R)$, so $l \in L \cap \pi^{-1}(R)$. Since $x \in X \cap R$ was arbitrary, we have $\pi(L \cap \pi^{-1}(R)) = X \cap R$. Since $\cc$ is closed under intersection with regular languages, $L \in \cc$, and $\pi^{-1}(R) \in \REG$, we conclude that $X \cap R \in \cc^\exists(M)$. 
\end{proof}

In the case of cones, when we additionally have closure under inverse homomorphism and arbitrary homomorphisms, by similar methods as the above proof we can easily obtain the following:

\begin{lemma}[Herbst {\cite[Lemma~4.1]{Herbst1991}}]\label{Lem:HerbstLemma4.1}
Let $\cc$ be a cone, let $M$ be a finitely generated monoid, and $R \in \REG^\exists(M)$. Then we have: 
\[
L \in \cc^\bullet(M) \implies LR, RL \in \cc^\bullet(M).
\]
\end{lemma}

In particular, for any cone $\cc$ and finitely generated group $G$, we see that if $G$ has word problem in $\cc$, then $\REG^\exists(G) \subseteq \cc^\forall(G)$, which thus ``flattens'' the square in the right-hand side of \eqref{Eq:squares-of-inclusions}.

\begin{lemma}[{Herbst \cite[Prop.~5.5(a)]{Herbst1991}}]\label{Lem:HerbstUniversal-subsets-preserved-by-finite-index}
Let $\cc$ be a cone, $G$ be a finitely generated group, and let $N \trianglelefteq_{\mathrm{f.i.}} G$ be a normal subgroup of finite index in $G$. Then:
\begin{enumerate}
\item  $\cc^\forall(N) \subseteq \cc^\forall(G)$. 
\item $\cc^\exists(G) \cap 2^N \subseteq \cc^\exists(N)$. 
\end{enumerate}
\end{lemma}

Having covered some of the basic properties of $\cc^\bullet$-subsets, most of which can be found already in \cite{Herbst1991}, this last lemma now leads us towards introducing two new concepts: relative $\cc^\bullet$\textit{-subsets}, which will then lead us naturally to $\cc^\bullet$\textit{-flatness}.

\subsection{Relative $\cc^\bullet$-subsets}\label{Subsec:relative-subsets}

Let, as before, $\cc$ be a class of languages, and let $M$ be a finitely generated monoid, with finite generating set $A$ and $\pi \colon A^\ast \to M$ surjective. Let $N \leq M$ be a finitely generated submonoid. Then $K \subseteq N$ is said to be a \textit{$N$-relative $\cc^\bullet$-subset} of $M$ if $K \in \cc^\bullet(M)$.  We denote the set of $N$-relative $\cc^\bullet$-subsets of $M$ by $\cc^\bullet(M \mid N)$. That is, we simply have
\begin{align}
\cc^\exists(M \mid N) &= \cc^\exists(M) \cap 2^N, \label{Def:existsrelative} \\
\cc^\forall(M \mid N) &= \cc^\forall(M) \cap 2^N. \label{Def:forallrelative}
\end{align}
The obvious question that now arises in this context is the following: \textit{what is the relationship between the two sets $\cc^\bullet(M \mid N)$ and $\cc^\bullet(N)$}? Here, we already have a distinction between the cases of $\cc^\forall$-subsets and $\cc^\exists$-subsets. We begin with a useful lemma in this line, which generalizes a result due to Herbst \cite[Corollary~4.4]{Herbst1991}.

\begin{lemma}\label{Lem:ccForall-Preserved-By-Submonoids2.4}
Let $\cc$ be a class of languages closed under inverse homomorphism. Let $M$ be a finitely generated monoid, $N \leqfg M$ a finitely generated submonoid. Let $X \subseteq M$. Then: 
\[
X \in \cc^\forall(M) \implies X \cap N \in \cc^\forall(N).
\]
\end{lemma}
\begin{proof}
Suppose $X \in \cc^\forall(M)$. Fix finite generating sets $A$ and $A'$ of $M$ and $N$, respectively, with surjective homomorphisms $\varphi \colon A^\ast \to M$ and $\varphi' \colon (A')^\ast \to N$. Consider the diagram in Lemma~\ref{Lem:Herbst-lemma-2.3}, taking $\tau \colon N \to M$ to be the inclusion mapping, and hence $X \cap N = \tau^{-1}(X)$. The lemma states that the diagram commutes and that $(\varphi')^{-1}(X \cap N) = h^{-1}(\varphi^{-1}(X))$. Since $X \in \cc^\forall(M)$, we have $\varphi^{-1}(X) \in \cc$. Now $\cc$ is closed under inverse homomorphism, so $h^{-1}(\varphi^{-1}(X)) \in \cc$, and hence the full pre-image of $X \cap N$ under $\varphi'$ is in $\cc$. That is, $X \cap N \in \cc^\forall(N)$. 
\end{proof}

We are now ready to show the first difference between relative $\cc^\forall$-subsets and relative $\cc^\exists$-subsets, in the form of the following lemma:

\begin{lemma}\label{Lem:Elementary-passing-to-submonoid-2.5+2.6}
Let $M$ be a finitely generated monoid, and let $N \leqfg M$ be a finitely generated submonoid. Let $\cc$ be a class of languages closed under inverse homomorphism. Then:
\begin{enumerate}
\item $\cc^\forall(M \mid N) \subseteq \cc^\forall(N)$, and
\item $\cc^\exists(M \mid N) \supseteq \cc^\exists(N)$.
\end{enumerate}
\end{lemma}
\begin{proof}
Note that (1) is just a restatement of Lemma~\ref{Lem:ccForall-Preserved-By-Submonoids2.4}, in view of the definition of a relative $\cc^\forall$-subset. Thus, we only have to prove (2). Fix arbitrary finite generating sets $A, B$ for $M$ resp.\ $N$, and let $\pi_M \colon A^\ast \to M$ resp.\ $\pi_N \colon B^\ast \to N$ be the associated surjective homomorphisms. Let $X \in \cc^\exists(N)$. Since $N$ is a finitely generated submonoid of $M$, there is some finite set $C \subseteq A^\ast$ of words such that $C$ generates $N$ as a submonoid of $N$. Let $\pi_C \colon C^\ast \to N$ be the associated surjective homomorphism, which thus factors through $\pi_M \colon A^\ast \to N$ via the inclusion $C \subseteq A^\ast$. As mentioned immediately before Lemma~\ref{Lem:Herbst-lemma-2.3}, it follows directly from this lemma that, since $\cc$ is closed under inverse homomorphism, that $X$ is also a $\cc^\exists$-subset of $N$ with respect to the generating set $C$. Hence there is some $L \subseteq C^\ast$ with $L \in \cc$ and $\pi_C(L) = X$. Since $\pi_C$ factors through $\pi_M$ via the inclusion $C \subseteq A^\ast$, we also have $\pi_M(L) = X$. Thus $X$ is a $\cc^\exists$-subset of $M$ entirely contained inside $N$, i.e.\ $X \in \cc^\exists(M \mid N)$ as desired. 
\end{proof}

In general, it is clear from the examples given in Example~\ref{Examples} that the inclusions in Lemma~\ref{Lem:Elementary-passing-to-submonoid-2.5+2.6} cannot be reversed. The bulk of this present article consists, indirectly or directly, of investigations into when the inclusions \textit{can} be promoted to equalities. For the remainder of this article, we will investigate this question from two points of view. First, we will investigate for which submonoids in Lemma~\ref{Lem:Elementary-passing-to-submonoid-2.5+2.6} the desired equalities hold. Next, we will introduce the notion of a $\cc^\bullet$-\textit{flat} monoid, being those monoids for which the equalities hold for \textit{all} finitely generated submonoids. We will describe some elementary properties of $\cc^\bullet$-flat monoids and groups in \S\ref{Subsec:flatness}, and in \S\ref{Sec:3flatmonoidsgeneral} we will investigate when $\cc^\bullet$-flatness is preserved under taking finite extensions and passing to finite index subgroups. In \S\ref{Sec:4flatgroupsforspecific}, we will discuss specific classes $\cc$, and investigate the classification of $\cc^\forall$-flat groups. 

Before discussing flatness, however, we will show some general properties of relative $\cc^\bullet$-subsets. We begin with a simple result.

\begin{proposition}\label{Prop:monoid-finite}
Let $\cc$ be a class of languages closed under inverse homomorphism. Let $M$ be a finitely generated monoid, and let $T \leq M$ be such that $M \setminus T$ is an ideal of $M$, and such that $|M \setminus T| < \infty$. Then  $\cc^\bullet(M \mid T) = \cc^\bullet(T)$.

\end{proposition}
\begin{proof}
Since $T$ is a large submonoid of $M$, there exists by \cite{Jura1978}, cf.\ also \cite[Theorem~1.1]{Ruskuc1998}, a finite generating set $B$ for $T$. Let $\pi_T \colon B^\ast \to T$ be the associated surjective homomorphism. Then $C = B \cup (M \setminus T)$ is a finite generating set for $M$, and there is a surjective homomorphism $\pi_M \colon C^\ast \to M$. Furthermore, since $M \setminus T$ is an ideal of $M$, we have that a word $w \in C^\ast$ represents an element of $T$ if and only if $w \in B^\ast$. Consequently, $\pi_M |_{B^\ast} = \pi_T$. 

We begin by proving the case of $\cc^\exists$-subsets. If $X \in \cc^\exists(M \mid T)$, then since $\cc$ is closed under inverse homomorphism there is some $L \subseteq B^\ast$ with $\pi_M(L) = X$. However, since $L \in B^\ast$, we have $\pi_T(L)= \pi_M(L) = X$, and hence $X \in \cc^\exists(T)$. The reverse inclusion is Lemma~\ref{Lem:Elementary-passing-to-submonoid-2.5+2.6}(2). Next, for $\cc^\forall$-subsets, we have that  $\cc^\forall(M \mid T) \subseteq \cc^\forall(T)$ by Lemma~\ref{Lem:Elementary-passing-to-submonoid-2.5+2.6}(1) and if $X \in \cc^\forall (T)$, then $\pi_C^{-1}(X)=\pi_T^{-1}(X)\in \cc$ and $X$ is entirely contained in $T$, thus belonging to $\cc^\forall(M \mid T)$.
\end{proof}

Recall that if $M$ is any monoid, then $M^0$ denotes the result of adjoining a zero element to $M$. Since $\{ 0 \}$ is an ideal of $M^0$, the following is true: 

\begin{corollary}\label{Cor:Preserved-by-adjoining-zero}
If $M$ is a finitely generated monoid, then $\cc^\bullet(M^0 \mid M) = \cc^\bullet(M)$. 
\end{corollary}

We now prove some of the group-theoretic theorems of this article, which gives some conditions on the subgroup $H$ which guarantee that the inclusions in Lemma~\ref{Lem:Elementary-passing-to-submonoid-2.5+2.6} can be promoted to equalities. The first is a generalization of Herbst's Lemma~\ref{Lem:HerbstUniversal-subsets-preserved-by-finite-index} from normal subgroups of finite index to arbitrary subgroups of finite index.

\begin{theorem}\label{Thm:Equalities-for-finite-index-subgroups-Thm2.8}
Let $G$ be a finitely generated group, and let $\cc$ be a full semi-$\AFL$. Then for all finite index subgroups $H \leqfi G$, we have: 
\begin{enumerate}[label=(\alph*)]
\item $\cc^\forall(G \mid H) = \cc^\forall(H)$, and \label{Thm2.8-a}
\item $\cc^\exists(G \mid H) = \cc^\exists(H)$\label{Thm2.8-b}
\end{enumerate}
\end{theorem}
\begin{proof}
Since $H \leqfi G$, there exists $F \leq H$ such that $F \trianglelefteq_{\mathrm{f.i.}} G$, and hence also $F \trianglelefteq_{\mathrm{f.i.}} H$. Let $b_1, \dots, b_n \in H$ be a transversal for $F$ in $H$, such that $H$ is decomposed as a union of disjoint right cosets $H = Fb_1 \cup \cdots \cup Fb_n$. 

We begin with \ref{Thm2.8-a}. By Lemma~\ref{Lem:Elementary-passing-to-submonoid-2.5+2.6}, it suffices to prove $\cc^\forall(H) \subseteq \cc^\forall(G \mid H)$. Let $K \in \cc^\forall(H)$ be arbitrary. We can write 
\begin{equation}\label{Eq:Thm2.8-union1}
K = K \cap H = \bigcup_{i=1}^n (F b_i \cap K).
\end{equation}
For ease of notation, let $K_i = F b_i \cap K$. Since $\cc$ is closed under union, being a full semi-$\AFL$, it suffices to prove that $K_i \in \cc$ for all $1 \leq i \leq n$. First, $K_i b_i^{-1} \subseteq F$, and since $F$ has finite index in $G$, it follows from Theorem~\ref{Thm:Anisimov-Seifert} that $F \in \REG^\forall(G)$. Hence $Fb_i \in \REG^\forall(G)$ by Lemma~\ref{Lem:HerbstLemma4.1}, and thus also $Fb_i \in \REG^\forall(H)$ by Lemma~\ref{Lem:Elementary-passing-to-submonoid-2.5+2.6}. Since $K \in \cc^\forall(H)$, it now follows from Lemma~\ref{Lem:closed-under-intersection-2.1} that $K_i = Fb_i \cap K \in \cc^\forall(H)$, and hence by another application of Lemma~\ref{Lem:HerbstLemma4.1} also $K_i b_i^{-1} \in \cc^\forall(H)$. Hence by Lemma~\ref{Lem:Elementary-passing-to-submonoid-2.5+2.6}, we have $K_i b_i^{-1} \in \cc^\forall(F)$, and hence since $F$ is normal and of finite index in $G$, by Lemma~\ref{Lem:HerbstUniversal-subsets-preserved-by-finite-index} we have $K_i b_i^{-1} \in \cc^\forall(G)$. By a final application of Lemma~\ref{Lem:HerbstLemma4.1}, we conclude $K_i \in \cc^\forall(G)$, as desired.

Next, we prove \ref{Thm2.8-b}. By Lemma~\ref{Lem:Elementary-passing-to-submonoid-2.5+2.6}, it suffices to prove $\cc^\exists(G \mid H) \subseteq \cc^\exists(H)$. Let $K \in \cc^\exists(G \mid H)$ be arbitrary, i.e.\ let $K \subseteq H$ be such that $K \in \cc^\exists(G)$. Decompose $K$ as in \eqref{Eq:Thm2.8-union1}, and note that it again suffices to prove $K_i \in \cc^\exists(H)$ for all $1 \leq i \leq n$. Arguing as in the case of $\cc^\forall$ above, several applications of Lemma~\ref{Lem:HerbstLemma4.1} and an application of the second half of Lemma~\ref{Lem:HerbstUniversal-subsets-preserved-by-finite-index} yields that $K_i \in \cc^\exists(H)$.
\end{proof}

As an application of Theorem~\ref{Thm:Equalities-for-finite-index-subgroups-Thm2.8} we resolve a question of Al Kohli, Bleak \& Elliott \cite{Kohli2025} regarding epi-$\cc$ groups. Let us recall their setting. For a class of languages $\cc$, they say that a group $G$ generated by a finite set $A$, with associated homomorphism $\pi \colon A^\ast \to G$, is \textit{epi-}$\cc$ if there is a language $L \in \cc$ such that $\pi(L) = G \setminus \{ 1 \}$. In the language of the present paper, this simply says that $G \setminus \{ 1 \} \in \cc^\exists(G)$. In \cite[Question~7.7]{Kohli2025}, the authors leave open the problem of whether the class of epi-$\cc$ groups is closed under passage to finite index subgroups. We are able to give a positive answer to this question for $\cc = \REG$ (a case already covered in \cite{Kohli2025}) as well as $\cc = \CF$ and $\cc = \RE$. More generally, we have the following:

\begin{corollary}\label{Cor:Answer-Collin-Question}
Let $\cc$ be any full semi-$\AFL$. Then the class of epi-$\cc$ groups is closed under passage to finite index subgroups.
\end{corollary}
\begin{proof}
Let $G$ be any epi-$\cc$ group, and let $H \leqfi G$ be a subgroup of finite index. By definition, we have $G \setminus \{1 \} \in \cc^\exists(G)$, and since $H$ has finite index in $G$ we have $H \in \REG^\forall(G)$ by Theorem~\ref{Thm:Anisimov-Seifert}. Thus, by Lemma~\ref{Lem:closed-under-intersection-2.1}, we have 
\[
H \setminus \{1 \} = (G \setminus \{ 1 \}) \cap H \in \cc^\exists(G)
\]
and consequently $H \setminus \{1 \} \in \cc^\exists(G \mid H)$. Since we have $\cc^\exists(G \mid H) = \cc^\exists(H)$ by Theorem~\ref{Thm:Equalities-for-finite-index-subgroups-Thm2.8}, we conclude that $H \setminus \{ 1 \} \in \cc^\exists(H)$. This is precisely saying that $H$ is epi-$\cc$. Since $G$ was arbitrary, we are done. 
\end{proof}

In particular, the classes of epi-$\REG$, epi-$\CF$, resp.\ epi-$\RE$ groups are all closed under passage to finite index subgroups. Moving beyond the Chomsky hierarchy, Corollary~\ref{Cor:Answer-Collin-Question} also shows, for example, that the same is true of the class of epi-$\ETOL$ groups. However, note that the class $\CS$ of context-sensitive languages is not a full semi-$\AFL$ (as it is not a cone), so our methods do not give an answer in this case.

\begin{remark}
In a second version of the preprint \cite{Kohli2025}, appearing after the first version of this present article, the authors present an argument by C.\ Bodart showing that the class of epi-$\cc$ is closed under taking finite index subgroups if $\cc\in\{\REG, \CF, \CS,\RE\}$ using coset-labelled words. In fact, we remark that following that same proof, one can extend Theorem  \ref{Thm:Equalities-for-finite-index-subgroups-Thm2.8} b), and hence also Corollary  \ref{Cor:Answer-Collin-Question}, to any trio $\cc$. 
\end{remark}

Having extended Herbst's result to arbitrary finite index subgroups, we now turn to proving a structural result for $\cc^\bullet$-subsets of a given group in terms of the $\cc^\bullet$-subsets of a finite index subgroup of the same. Grunschlag \cite{Grunschlag} and Silva \cite{Silva} independently obtained the following result for the case of $\REG^\bullet$, and Carvalho \cite{Carvalho2023} did the same for $\CF^\bullet$. We provide a generalization of their results to all cones $\cc$ closed under union, i.e.\ to all full semi-$\AFL$s $\cc$. Precisely, we have the following:

\begin{theorem}\label{Thm:Structure-of-cc-subsets-2.9}
Let $\cc$ be a full semi-$\AFL$. Let $G$ be a finitely generated group, and let $H \leqfi G$. Let $b_1, \dots, b_n$ be a right transversal for $H$ in $G$. Then 
\begin{equation}\label{Eq:description-of-subsets}
\cc^\bullet(G) = \left\{ \bigcup_{i=1}^n L_i b_i \biggm| L_i \in \cc^\bullet(H) \text{ for all $1 \leq i \leq n$} \right\}.
\end{equation}
\end{theorem}
\begin{proof}
The proof works the same both in the case of $\cc^\forall(G)$ and $\cc^\exists(G)$, so we maintain the ambiguous notation $\cc^\bullet(G)$ throughout the proof. We first prove the inclusion $\supseteq$ of \eqref{Eq:description-of-subsets}. Suppose that $X= \bigcup_{i=1}^n L_i b_i$ where $L_i \in \cc^\bullet(H)$ for all $1 \leq i \leq n$. By Theorem \ref{Thm:Equalities-for-finite-index-subgroups-Thm2.8} we have $L_i \in \cc^\bullet(G)$, and thus also $L_ib_i \in \cc^\bullet(G)$ for all $1 \leq i \leq n$ by Lemma \ref{Lem:HerbstLemma4.1}. Since $\cc$ is closed under union, we have $X\in \cc^\bullet(G)$, as desired.

To prove the inclusion $\subseteq$, let $X\in \cc^\bullet(G)$. As in the proof of Theorem~\ref{Thm:Equalities-for-finite-index-subgroups-Thm2.8}, we can write $X=X\cap G=\bigcup_{i=1}^n (Hb_i\cap X)$, and since $H\leqfi G$, it follows from Theorem~\ref{Thm:Anisimov-Seifert} that $H\in \REG^\forall (G)$. By Lemma~\ref{Lem:HerbstLemma4.1}, it follows that $Hb_i\in \REG^\forall (G)$ for all $1 \leq i \leq n$. Set $X_i= Hb_i\cap X$. Then $X_i\in \cc^\bullet(G)$ by Lemma~\ref{Lem:closed-under-intersection-2.1}, and hence by Lemma~\ref{Lem:HerbstLemma4.1} we have $L_i=K_ib_i^{-1}\in \cc^\bullet(G)$. Since $L_i \subseteq H$, we can apply Theorem~\ref{Thm:Equalities-for-finite-index-subgroups-Thm2.8}, to deduce that $L_i\in \cc^\bullet (H)$, and thus that
\[
X=\bigcup_{i=1}^n K_i=\bigcup_{i=1}^n L_ib_i,
\]
and since $X \in \cc^\bullet(G)$ was arbitrary, this completes the proof. 
\end{proof}

As a corollary, we can deduce the following generalization of \cite[Corollary~3.5]{Carvalho2023} and \cite[Corollary~3.9]{Carvalho2023} from the class of context-free languages to arbitrary full semi-$\AFL$s.

\begin{corollary}\label{Cor:cc-bullet-closed-under-int-and-complement-Cor2.10}
Let $\cc$ be a full semi-$\AFL$. Let $G$ be a finitely generated group, and let $H \leqfi G$. Then $\cc^\bullet(G)$ is closed under intersection (resp.\ complement) if and only if $\cc^\bullet(H)$ is closed under intersection (resp.\ complement). 
\end{corollary}
\begin{proof}
Let $b_1,\ldots, b_n$ be a right transversal for $H$ in $G$. We start by proving the claim about intersections. Suppose that $\cc^\bullet(G)$ is closed under intersection and let $X,Y\in \cc^\bullet(H)$. From Theorem~\ref{Thm:Equalities-for-finite-index-subgroups-Thm2.8}, we know that $X,Y\in \cc^\bullet(G)$, so $X\cap Y\in \cc^\bullet(G)$ and, again by Theorem~\ref{Thm:Equalities-for-finite-index-subgroups-Thm2.8}, we have that  $X\cap Y\in \cc^\bullet(H)$. Suppose now that  $\cc^\bullet(H)$ is closed under intersection and take $X,Y\in \cc^\bullet(G)$. From Theorem~\ref{Thm:Structure-of-cc-subsets-2.9}, there exist subsets $X_i, Y_i\in \cc^\bullet(H)$ such that 
\begin{align*}
X=\bigcup_{i=1}^n X_ib_i \quad \text{ and } \quad Y=\bigcup_{i=1}^n Y_ib_i.
\end{align*}
Hence, we have that
\begin{align}
\label{eq: intersection}
X\cap Y=\left(\bigcup_{i=1}^n X_ib_i\right)\cap \left(\bigcup_{i=1}^n Y_ib_i\right)=\bigcup_{i=1}^n (X_i\cap Y_i)b_i.
\end{align}
Since $X_i\cap Y_i\in \cc^\bullet(H)$, for all $1 \leq i \leq n$, by Theorem~\ref{Thm:Equalities-for-finite-index-subgroups-Thm2.8} also  $X_i\cap Y_i\in \cc^\bullet(G)$. By Lemma~\ref{Lem:HerbstLemma4.1}, we now conclude $(X_i\cap Y_i)b_i\in \cc^\bullet(G).$ Since $\cc$ is a full semi-AFL, it is closed under union, which, together with \eqref{eq: intersection} implies that $X\cap Y\in \cc^\bullet(G)$.

We will now deal with the complement case. Suppose that $\cc^\bullet(G)$ is closed under complement and let $X\in \cc^\bullet(H)$. By Theorem~\ref{Thm:Equalities-for-finite-index-subgroups-Thm2.8}, $X\in \cc^\bullet(G)$, and so $G\setminus X\in \cc^\bullet(G)$. Applying Theorem~\ref{Thm:Structure-of-cc-subsets-2.9}, we get that there are subsets $L_i\in \cc^\bullet(H)$ such that $G\setminus X=\bigcup_{i=1}^n L_ib_i$.  Letting $1 \leq k \leq n$ be the unique $k$ such that $b_k\in H$, we obtain that
\[
H\setminus X=(G\setminus X)\cap H=H\cap \bigcup_{i=1}^n L_ib_i=L_jb_j\in \cc^\bullet(H),
\]
by Lemma~\ref{Lem:HerbstLemma4.1}, as required.

Suppose now that $\cc^\bullet(H)$ is closed under complement and let $X\in \cc^\bullet(G)$. Then $X=\bigcup_{i=1}^n X_ib_i$ for some $X_i\in \cc^\bullet(H)$, by Theorem~\ref{Thm:Structure-of-cc-subsets-2.9} and so 
\begin{align}
\label{eq: complement}
G\setminus X=\bigcup_{i=1}^n (H\setminus X_i)b_i.
\end{align} 
Since $\cc^\bullet(H)$ is closed under complement and $X_i\in \cc^\bullet(H)$, then $H\setminus X_i\in \cc^\bullet(H)$. By Theorem~\ref{Thm:Equalities-for-finite-index-subgroups-Thm2.8},  we get that $H\setminus X_i\in \cc^\bullet(G)$, and so $(H\setminus X_i)b_i\in \cc^\bullet(G)$ by Lemma~\ref{Lem:HerbstLemma4.1}. Since $\cc$ is a full semi-AFL, it is closed under union, which, by (\ref{eq: complement}), implies that $G\setminus X\in\cc^\bullet(G)$.
\end{proof}

We now turn to using these properties of relative $\cc^\bullet$-subsets. In particular, we will apply them in our study of a new property of finitely generated groups: namely, that of $\cc^\bullet$\textit{-flatness}.

\subsection{$\cc^\bullet$-flatness}\label{Subsec:flatness}

As indicated after Lemma~\ref{Lem:Elementary-passing-to-submonoid-2.5+2.6}, we now initiate the second part of our investigation into when the inequalities of Lemma~\ref{Lem:Elementary-passing-to-submonoid-2.5+2.6} can be promoted to equalities. In this line, we introduce the following terminology:  

\begin{definition}
Let $M$ be a finitely generated monoid, and $\cc$ a class of languages closed under inverse homomorphism. We say that $M$ is $\cc^\bullet$\textit{-flat\footnote{The terminology \textit{flat} is derived from homological algebra, by analogy with the condition under which a change of rings is well-behaved.} (as a monoid)} if for every finitely generated submonoid $N \leqfg M$ we have $\cc^\bullet(M \mid N) = \cc^\bullet(N)$. If $G$ is a group, then we analogously we say that $G$ is $\cc^\bullet$\textit{-flat (as a group)} if for every finitely generated subgroup $H \leqfg G$ we have $\cc^\bullet(G \mid H) = \cc^\bullet(H)$. 
\end{definition}

If $G$ is a group, then, unless explicitly specified otherwise, $\cc^\bullet$-flatness will refer to $\cc^\bullet$-flatness as a group, i.e.\ we will only consider finitely generated subgroups of $G$ rather than general submonoids; cf.\ also Remark~\ref{Rem:submonoid-flatness-vs-subgroup-flatness}. 

We begin with some examples and initial notes on $\cc^\bullet$-flat groups. First, for example, the following result is a simple restatement of Proposition~\ref{Prop:All-REGe-flat-old} in our new language.

\begin{proposition}[Anisimov \& Seifert, 1975 \cite{Anisimov1975}]
All groups are $\REG^\exists$-flat. \label{Prop:AnisimovSeifert-rephrased-as-flatness}
\end{proposition}

This is called a kind of ``Fatou property'' for groups by Berstel \& Sakarovitch \cite{Berstel1986}. The question of characterizing precisely which groups are $\cc^\exists$-flat seems very difficult in general. Herbst \cite{Herbst1992} proved that all virtually free groups are $\CF^\exists$-flat, although an example due to Carvalho \cite[Example~4.8]{Carvalho2023} shows that not all groups are $\CF^\exists$-flat. We also have the following general result, also proving that the class of $\CF^\exists$-flat groups is larger than the class of virtually free groups. 

\begin{proposition}\label{Prop:AllAbelianAreCFE-flat}
Any finitely generated virtually abelian group is $\CF^\exists$-flat.
\end{proposition}
\begin{proof}
It is not difficult to show (cf.\ \cite[p.~260]{Herbst1991}) that in any finitely generated virtually abelian group $G$ we have $\CF^\exists(G) = \REG^\exists(G)$. Thus in any finitely generated virtually abelian group $G$ with $H \leqfg G$ we have 
\[
\CF^\exists(G \mid H) = \CF^\exists(G) \cap 2^H = \REG^\exists(G) \cap 2^H = \REG^\exists(G \mid H) = \REG^\exists(H)
\]
using Proposition~\ref{Prop:AnisimovSeifert-rephrased-as-flatness} for the final equality. Since the subgroup $H$ is also finitely generated virtually abelian, we have $\REG^\exists(H) = \CF^\exists(H)$, and we are done. 
\end{proof}

Thus the class of $\CF^\exists$-flat groups strictly contains all virtually free groups and all abelian groups, but does not contain all solvable groups by virtue of \cite[Example~4.8]{Carvalho2023}. The following question is tantalizing, and clearly illustrates the difficulty of working with $\CF^\exists$-flatness in general. 

\begin{question}\label{question: characterize cfflat}
Is there an algebraic characterization of $\CF^\exists$-flat groups?
\end{question}

On the other hand, characterizing $\cc^\forall$-flat groups seems like a tractable problem in many cases, as illustrated by the following results.

\begin{theorem}\label{Thm:REG-forall-CF-forall-characterized}
Let $G$ be a finitely generated group. Then:
\begin{enumerate}
\item $G$ is $\REG^\forall$-flat if and only if $G$ is finite (Anisimov \& Seifert \cite{Anisimov1975}).
\item $G$ is $\CF^\forall$-flat if and only if $G$ is virtually free (Carvalho, \cite{Carvalho2023}).
\end{enumerate}
\end{theorem}

One direction of the implications in Theorem~\ref{Thm:REG-forall-CF-forall-characterized} is a consequence of the classical Theorem~\ref{Thm:AnisimovMSHerbst} combined with the following more general result.

\begin{lemma}\label{Lem:If-CAflat-then-WP-C}
Let $\cc$ be a class of languages closed under inverse homomorphism, and let $G$ be a finitely generated group. If $G$ is $\cc^\forall$-flat, then the word problem of $G$ lies in $\cc$. 
\end{lemma}
\begin{proof}
Suppose that $G$ is $\cc^\forall$-flat. Then 
\[
\{ 1 \} \in \cc^\forall(1) = \cc^\forall(G \mid 1) \subseteq \cc^\forall(G),
\]
where the only non-trivial step is the equality $\cc^\forall(1) = \cc^\forall(G \mid 1)$, which follows from the definition of $\cc^\forall$-flatness. However, $\{ 1 \} \in \cc^\forall(G)$ means precisely that the word problem of $G$ lies in $\cc$, so we are done. 
\end{proof}

Combining Lemma~\ref{Lem:If-CAflat-then-WP-C} with Lemma~\ref{Lem:HerbstLemma4.1}, it is easy to see that if $\cc$ is closed under union, then a group $G$ has word problem in $\cc$ if and only if for all subgroups $H \leqfg G$ we have 
\[
\cc^\forall(H) \cap \Fin(H) = \cc^\forall(G) \cap \Fin(H),
\]
where $\Fin(H)$ denotes the set of all finite subsets of $H$. 

Now, Lemma~\ref{Lem:If-CAflat-then-WP-C} combined with Theorem~\ref{Thm:REG-forall-CF-forall-characterized} naturally suggests the following tempting conjecture, which would strongly connect our new notion of $\cc^\forall$-flatness with the familiar realm of the word problem. 

\begin{conjecture}\label{Conj:Main-false-conjecture}
Let $G$ be a finitely generated group, and $\cc$ a class of languages closed under inverse homomorphism. Then $G$ is $\cc^\forall$-flat if and only if the word problem of $G$ lies in $\cc$.
\end{conjecture}

This conjecture, first posed by the second author to the first, thus holds true when $\cc = \REG$ and $\cc = \CF$. However, to the authors' surprise, the conjecture turns out to be \textbf{false} in general; and the present article partially grew out of an attempt to understand precisely to what extent the conjecture fails, and why it nevertheless holds true for the classes $\REG$ and $\CF$ of regular and context-free languages, respectively. Counterexamples to Conjecture~\ref{Conj:Main-false-conjecture} and further discussion will be the subject of \S\ref{Sec:4flatgroupsforspecific}. Before proceeding with this, however, we will present some general results on $\cc^\bullet$-flat groups and monoids in \S\ref{Sec:3flatmonoidsgeneral}.

\section{Closure properties of $\cc^\bullet$-flat monoids and groups}\label{Sec:3flatmonoidsgeneral}

\noindent In this section, we prove some general results on $\cc^\bullet$-flat monoids and groups. We begin with a result along the lines of Proposition~\ref{Prop:monoid-finite}, showing that $\cc^\bullet$-flatness is preserved by taking finitely generated submonoids resp.\ subgroups.

\begin{proposition}\label{Prop:Flatness-preserved-by-sub}
Let $\cc$ be a class of languages closed under inverse homomorphism. Let $M$ be a finitely generated monoid and $N \leqfg M$ be a finitely generated submonoid. If $M$ is $\cc^\bullet$-flat, then $N$ is $\cc^\bullet$-flat. Analogously, let $G$ be a finitely generated group, and $H \leqfg G$ be a finitely generated subgroup. If $G$ is $\cc^\bullet$-flat (as a group), then $H$ is $\cc^\bullet$-flat (as a group).
\end{proposition}
\begin{proof}
We start by proving the claim for $\cc^\forall$-flatness. Suppose that $M$ is   $\cc^\forall$-flat and let $T\leqfg N$. We want to prove that 
$\cc^\forall(N \mid T)=\cc^\forall(T).$
 In view of   Lemma~\ref{Lem:Elementary-passing-to-submonoid-2.5+2.6}(1), it suffices to show that  $\cc^\forall(T)\subseteq\cc^\forall(N \mid T)$.
 Let $X \in  \cc^\forall(T)$. Since $G$ is  $\cc^\forall$-flat, then $\cc^\forall(T)=\cc^\forall(M \mid T)$, and so $X\in \cc^\forall(M)$. Since $X$ is fully contained in $T\leqfg N$, that is, $X\in \cc^\forall(M \mid N)$ it follows by another application of  Lemma~\ref{Lem:Elementary-passing-to-submonoid-2.5+2.6}(1) that, $X\in \cc^{\forall}(N)$, thus  $X\in \cc^{\forall}(N\mid T)$.
 
 Now, we deal with the case of $\cc^\exists$-flatness. Suppose that $M$ is   $\cc^\exists$-flat and let $T\leqfg N$. In view of Lemma~\ref{Lem:Elementary-passing-to-submonoid-2.5+2.6}(2), it suffices to show that  $\cc^\exists(N \mid T)\subseteq \cc^\exists(T)$. Let $X\in \cc^\exists(N \mid T)$. Then $X\in \cc^\exists(N)$, and so, by  Lemma~\ref{Lem:Elementary-passing-to-submonoid-2.5+2.6}(2),  $X\in \cc^\exists(M \mid T)$.
Since $M$ is  $\cc^\exists$-flat,  then  $ \cc^\exists(M \mid T)=\cc^\exists(T)$ and so $X\in \cc^\exists(T)$.
\end{proof}

For another preservation property, we have the following, showing that at least $\cc^\exists$-flatness behaves well with respect to taking large submonoids. 

\begin{proposition}
Let $\cc$ be a full semi-$\AFL$.
Let $M$ be a finitely generated submonoid, and let  $N \leq M$ be such that $M \setminus N$ is an ideal of $M$, and such that $|M \setminus N| < \infty$.Then $M$ is $\cc^\exists$-flat if and only if $N$ is $\cc^\exists$-flat.
\end{proposition}
\begin{proof}
Since flatness is preserved by taking submonoids by Proposition \ref{Prop:Flatness-preserved-by-sub}, we only have to prove that if $N$ is $\cc^\exists$-flat, then so is $M$. Hence, suppose that $N$ is $\cc^\exists$-flat and let $T\leqfg M$. In order to prove that $N$ is $\cc^\exists$-flat, it is enough to show that  $\cc^\exists(M \mid T) \subseteq  \cc^\exists(T)$ by Lemma \ref{Lem:Elementary-passing-to-submonoid-2.5+2.6}(2). Let $X\in \cc^\exists(M \mid T)$. Then $X=X_N\cup X_{\overline N}$, where $X_N=X\cap  N$ and $X_{\overline N}=X\setminus X_N\subseteq M\setminus N$. Since $\cc$ is closed under union, so is $\cc^\exists(T)$ and since $\cc$ contains all singleton languages, it follows that $X_{\overline N}\in \cc^\exists(T)$, since $X_{\overline N}$ is finite. We will now prove that $X_N\in \cc^\exists(T)$, which, together with closure under union, implies that $X\in \cc^\exists(T)$. To do so, we proceed as in the proof of Proposition \ref{Prop:monoid-finite}.

Take a finite generating set $B$ for $N$ and let $\pi_N \colon B^\ast \to N$ be the associated surjective homomorphism. Then $C = B \cup (M \setminus N)$ is a finite generating set for $M$, and there is a surjective homomorphism $\pi_M \colon C^\ast \to M$. 
A word $w \in C^\ast$ represents an element of $N$ if and only if $w \in B^\ast$, and hence $\pi_M |_{B^\ast} = \pi_N$. Since $X\in \cc^\exists(M)$, there exists a language $L\in \cc$ such that $L\pi_M=X$. Since $\cc$ is a full semi-AFL, it is closed under intersection with regular languages, thus $L_N=L\cap B^*\in \cc$ and $X_N=L_N\pi_M|{B^\ast}=L_N\pi_N$. Therefore, $X_N\in \cc^\exists(N \mid N\cap T)$, which by $\cc^\exists$-flatness of $N$, implies that $X_N\in \cc^\exists(N\cap T)$. It now follows from Lemma \ref{Lem:Elementary-passing-to-submonoid-2.5+2.6}(2) that  $X_N\in \cc^\exists(T)$.
\end{proof}

We now prove that $\cc^\forall$-flatness as a group is preserved by taking finite extensions; we do not know if the corresponding statement holds for $\cc^\exists$-flatness, see Question~\ref{Quest:Is-ccE-flatness-closed-wrt-finite-extensions}.

\begin{proposition}\label{Prop:cc-flatness-preserved-by-finite-index}
Let $\cc$ be a full semi-$\AFL$. Let $G$ be a finitely generated group, and  let $H \leqfi G$. Then $G$ is $\cc^\bullet$-flat if and only if $H$ is $\cc^\bullet$-flat. 
\end{proposition}
\begin{proof}
The direction $(\implies)$ follows directly from Proposition~\ref{Prop:Flatness-preserved-by-sub}. We now prove the direction $(\impliedby)$ for the universal case. Suppose that $H$ is $\cc^\forall$-flat and take $H'\leqfg G$. Since $H\leqfi G$, we have $H\cap H'\leqfi H'$. Let $b_1,\ldots,b_n$ be a right transversal for $H\cap H'$ in $H'$.  We want to prove that $\cc^\forall(H')\subseteq\cc^\forall(G\mid H')$, and that suffices by   Lemma~\ref{Lem:Elementary-passing-to-submonoid-2.5+2.6}(1). 
Let $X\in  \cc^\forall(H')$ and write  
$$X=X\cap H'=\bigcup_{i=1}^n \left((H'\cap H)b_i\cap X\right).$$ Now,  put $X_i=(H\cap H')b_i \cap X$. We will prove that for all $1 \leq i \leq n$ we have $X_i\in \cc^\forall(G)$. Together with the fact that $\cc$ is closed under union, this will yield the result. 

Since $H\cap H'\leqfi H'$, then $H\cap H'\in \REG^\forall(H')$ by ~Theorem~\ref{Thm:Anisimov-Seifert}. By Lemma ~\ref{Lem:HerbstLemma4.1}, $(H\cap H')b_i\in \REG^\forall(H')$. Then, by Lemma~\ref{Lem:closed-under-intersection-2.1}, $X_i\in \cc^\forall(H')$, and, again by Lemma ~\ref{Lem:HerbstLemma4.1}, $X_ib_i^{-1}\in \cc^\forall(H')$. Since  $X_ib_i^{-1}\subseteq H\cap H'\leqfi H'$, it now follows from Theorem~\ref{Thm:Equalities-for-finite-index-subgroups-Thm2.8} that $X_ib_i^{-1}\in \cc^\forall(H\cap H')$. Since $H$ is  $\cc^\forall$-flat, $X_ib_i^{-1}\in \cc^\forall (H)$ and by Theorem~\ref{Thm:Equalities-for-finite-index-subgroups-Thm2.8}, we have that $X_ib_i^{-1}\in \cc^\forall (G)$. Therefore, by Lemma~\ref{Lem:HerbstLemma4.1}, $X_i\in \cc^\forall (G)$. 

Finally, we prove the direction $(\impliedby)$ for the existencial case. Let $H\leqfi G$ be such that $H$ is $\cc^\exists$-flat and $H'\leqfg G$. We want to prove that $G$ is $\cc^\exists$-flat, that is, that $\cc^\exists(H')=\cc^\exists(G\mid H')$. In view of Lemma \ref{Lem:Elementary-passing-to-submonoid-2.5+2.6}, the only thing to prove is that $\cc^\exists(G\mid H')\subseteq \cc^\exists(H')$. Let then $X\in \cc^\exists(G\mid H')$ and  $\{b_1,\ldots b_n\}$ be a right transversal for $H$ in $H\cap H'$. Then $$X=\bigcup_{i=1}^n X\cap (H\cap H')b_i=\bigcup_{i=1}^n X\cap Hb_i.$$ Indeed, the first equality follows from the fact that $X\subseteq H'$, and so $X=X\cap H'$. For the second equality, it is obvious that, for all $i\in [n]$, $X\cap (H\cap H')b_i\subseteq X\cap Hb_i$ and if $g\in X\cap Hb_i$, then $g=hb_i$ for some $h\in H$ and $hb_i\in X\subseteq H'$. Since $b_i\in H'$, then $h\in H'$ and $g=hb_i$, with $h\in H\cap H'$.

Now, $H\leqfi G,$ which, by Theorem \ref{Thm:Anisimov-Seifert}, means that $H\in \REG^\forall(G)$. Let $i\in [n]$. By Lemma \ref{Lem:HerbstLemma4.1}, it follows that $Hb_i\in \REG^\forall(G)$, and so, putting $X_i= X\cap Hb_i$, we have that $X_i\in \cc^\exists(G)$ by Lemma \ref{Lem:closed-under-intersection-2.1}. Now, $X_ib_i^{-1}$ must belong to $\cc^\exists(G\mid H)$ by Lemma \ref{Lem:HerbstLemma4.1}. Since $H\leqfi G$, Theorem \ref{Thm:Equalities-for-finite-index-subgroups-Thm2.8} yields that $X_ib_i^{-1}\in \cc^\exists(H)$. Since $X_ib_i^{-1}\subseteq H'$ and $H$ is $\cc^\exists$-flat, then $X_ib_i^{-1}\in \cc^\exists(H')$, which, again by Lemma \ref{Lem:HerbstLemma4.1}, means that $X_i\in \cc^\exists(H')$. Since $\cc$ is assumed to be closed under union, then so is $\cc^\exists(H')$, and thus $X=\bigcup_{i=1}^n X_i\in \cc^\exists(H')$.

This completes the proof of the lemma. 
\end{proof}

\begin{remark}\label{Rem:submonoid-flatness-vs-subgroup-flatness}
Let $\REC$ be the class of all recursive languages. In \S\ref{Subsec:Rec-and-CS}, we will show that a group is $\REC^\forall$-flat \textit{as a group} if and only if it has decidable subgroup membership problem; and that a group is $\REC^\forall$-flat \textit{as a monoid} if and only if it has decidable submonoid membership problem. Decidability of the subgroup membership problem is clearly preserved by taking finite extensions. On the other hand, Dong \cite{Dong2024} has recently shown that there exists a group $G$ with a finite index subgroup $H \leqfi G$ such that the submonoid membership problem is decidable in $H$ but undecidable in $G$. This shows that $\REC^\forall$-flatness \textit{as a monoid} is not in general preserved by taking finite extensions of groups; but Proposition~\ref{Prop:cc-flatness-preserved-by-finite-index} shows that $\REC^\forall$-flatness \textit{as a group} is, on the other hand, preserved by finite extensions of groups. In particular, there exists a group which is $\REC^\forall$-flat as a group, but not as a monoid. This shows the importance of distinguishing between these types of flatness. 
\end{remark}

Thus, combining Proposition~\ref{Prop:Flatness-preserved-by-sub} and Proposition~\ref{Prop:cc-flatness-preserved-by-finite-index}, we deduce the following.

\begin{theorem}\label{Thm:cc-forall-is-closed-under-fg-and-fi}
Let $\cc$ be a full semi-$\AFL$. Then the class of $\cc^\forall$-flat groups is closed under taking finitely generated subgroups and finite extensions. 
\end{theorem}

Finally, to end this section, we mention that as discussed above, we do not know whether the analogue of Proposition~\ref{Prop:cc-flatness-preserved-by-finite-index} holds for $\cc^\exists$-flat groups. We therefore pose this natural question.

\begin{question}\label{Quest:Is-ccE-flatness-closed-wrt-finite-extensions}
If $\cc$ is a full semi-$\AFL$, is $\cc^\exists$-flatness preserved by taking finite extensions of groups? In particular, is the class of $\CF^\exists$-flat groups closed under taking finite extensions? 
\end{question}

Since every finitely generated group is $\REG^\exists$-flat by Proposition~\ref{Prop:AnisimovSeifert-rephrased-as-flatness}, it follows that Question~\ref{Quest:Is-ccE-flatness-closed-wrt-finite-extensions} at least has an affirmative answer in the case of $\cc = \REG$.

\section{Characterizing $\cc^\bullet$-flat groups for specific classes}\label{Sec:4flatgroupsforspecific}

\noindent In this section, we will discuss the problem of characterizing the class of $\cc^\bullet$-flat groups for specific classes of languages $\cc$. Given the apparent general difficulty of describing $\cc^\exists$-flat groups, we will mainly focus on providing characterizations of $\cc^\forall$-flat groups, and in particular provide further examples of when Conjecture~\ref{Conj:Main-false-conjecture} is true beyond Theorem~\ref{Thm:REG-forall-CF-forall-characterized}, and provide counterexamples to the conjecture. First, in \S\ref{Subsec:Characterizing-1Counter}, we show that the $\OC^\forall$-groups can be completely characterized as the virtually cyclic groups. Next, in \S\ref{Subsec:Rec-and-CS}, we characterize all $\REC^\forall$-flat groups and monoids (with flatness taken as groups resp.\ as monoids, keeping in mind Remark~\ref{Rem:submonoid-flatness-vs-subgroup-flatness}) as those with a decidable subgroup resp.\ submonoid membership problem. This shows that the class of $\REC^\forall$-flat groups is strictly smaller than the class of groups with word problem in $\REC$, thus providing our first counterexample to Conjecture~\ref{Conj:Main-false-conjecture}. We also show that classes $\cc = \CS$ and $\cc = \pCF$ (the context-sensitive and poly-context-free languages, respectively) also yield counterexamples. Finally, in \S\ref{Subsec:Recursively-enumerable-flatness}, we discuss the case of $\cc = \RE$, the recursively enumerable languages, and characterize the $\RE^\forall$-flat groups, verifying Conjecture~\ref{Conj:Main-false-conjecture} for this class. Finally, using Tarski monsters, we prove that the class of $\RE^\exists$-flat groups is strictly contained in the class of $\RE^\forall$-flat groups.

\subsection{One-counter flatness}\label{Subsec:Characterizing-1Counter}

We now show that Conjecture~\ref{Conj:Main-false-conjecture} has an affirmative answer at least in the case of $\cc = \OC$, i.e.\ the class of $\OC^\forall$-flat groups can be completely characterized as those groups with a one-counter word problem. 

\begin{theorem}\label{Thm:OC-forall-characteriztion.}
Let $G$ be a finitely generated group. The following are equivalent:
\begin{enumerate}
\item $G$ is $\OC^\forall$-flat.
\item $G$ has a one-counter word problem.
\item $G$ is virtually cyclic.
\end{enumerate}
\end{theorem}
\begin{proof}
The equivalence of (2) and (3) is due to Herbst \cite{Herbst1991}. By Lemma~\ref{Lem:If-CAflat-then-WP-C}, we have (1) $\implies$ (2). Thus, we prove (3) $\implies$ (1), i.e.\ we prove that any virtually cyclic group is $\OC^\forall$-flat. Since the trivial group is clearly $\OC^\forall$-flat, it suffices in view of Proposition~\ref{Prop:cc-flatness-preserved-by-finite-index} to prove that $\Z$ is $\OC^\forall$-flat. Thus, let $H \leqfg \Z$. If $H$ is trivial, then $\OC^\forall(H) = \{ 1 \} =  \OC^\forall(\Z \mid H)$, since $\{ 1 \} \in \OC^\forall(\Z)$. If $H$ is non-trivial, then it has finite index in $\Z$, and hence by Theorem~\ref{Thm:Equalities-for-finite-index-subgroups-Thm2.8} we have $\OC^\forall(\Z \mid H) = \OC^\forall(H)$. Since $H$ was arbitrary, we are done. 
\end{proof}

On the other hand, the question of characterizing $\OC^\exists$-flat groups also remains an open problem. We have, at least, the following result in this line. 

\begin{proposition}\label{Prop:virtually-abelian-OC-flat}
Every virtually abelian group is $\OC^\exists$-flat. 
\end{proposition}

The proof of this proposition works identically to the proof of Proposition~\ref{Prop:AllAbelianAreCFE-flat}, since in any virtually abelian group we have $\REG^\exists(G) = \CF^\exists(G)$, it follows that also $\REG^\exists(G) = \cc^\exists(G)$ for any class $\cc$ with $\REG \subseteq \cc \subseteq \CF$. One such class is, of course, $\cc = \OC$, which yields the above result. As an aside, we remark that Greibach \cite{Greibach1969} has proved that there exists an infinite hierarchy 
\[
\REG \subsetneq \cc_1 \subsetneq \cc_2 \subsetneq \cdots \subsetneq \cc_i \subsetneq \cdots \subsetneq \CF
\]
where each $\cc_i$ is a full $\AFL$s, yielding infinitely many variations of Proposition~\ref{Prop:virtually-abelian-OC-flat}.

Combining Theorem~\ref{Thm:OC-forall-characteriztion.} and Proposition~\ref{Prop:virtually-abelian-OC-flat}, we deduce the following result, giving a source of $\OC^\exists$-flat groups.

\begin{proposition}\label{Prop:OC-forall=>OC-exists}
Every $\OC^\forall$-flat group is $\OC^\exists$-flat.
\end{proposition}

In other words, every virtually cyclic group is $\OC^\exists$-flat. The converse, however, fails by Proposition~\ref{Prop:virtually-abelian-OC-flat}, since e.g.\ $\Z^2$ is $\OC^\exists$-flat but is not virtually cyclic and consequently also not $\OC^\forall$-flat.

\subsection{Recursive and context-sensitive flatness}\label{Subsec:Rec-and-CS}

We now present our first counterexamples to our Conjecture~\ref{Conj:Main-false-conjecture}, by proving that $\cc^\forall$-flatness is in general a strictly stronger property than having word problem in $\cc$. The general construction which leads to our counterexamples is the following proposition: 

\begin{proposition}\label{Prop:Counterexample-machine}
Let $\cc$ be a class of languages closed under inverse homomorphism, and let $G$ be a finitely generated group. Assume the following all hold: 
\begin{enumerate}[label=(\alph*)]
\item $\cc \subseteq \REC$, i.e.\ every language in $\cc$ is recursive. 
\item The word problem for $G$ lies in $\cc$.
\item There is a fixed finitely generated subgroup $H \leqfg G$ such that the membership problem for $H$ in $G$ is undecidable.  
\end{enumerate}
Then $G$ is not $\cc^\forall$-flat (as a group). 
\end{proposition}
\begin{proof}
Suppose that all three specified conditions hold. Let $\pi \colon A^\ast \to G$ be an arbitrary surjection with $A$ a finite set, and fix a subgroup $H \leqfg G$ such that membership for $H$ in $G$ is undecidable. Then $\pi^{-1}(H) \not\in \REC$, for otherwise we could decide membership in $H$. Since $\cc \subseteq \REC$, also $\pi^{-1}(H) \not\in \cc$, and in particular $H \not\in \cc^\forall(G)$ which \textit{a fortiori} yields $H \not\in \cc^\forall(G \mid H)$. However, trivially $H \in \cc^\forall(H)$, and hence $\cc^\forall(G \mid H) \neq \cc^\forall(H)$. Thus $G$ is not $\cc^\forall$-flat. 
\end{proof}

By considering the special case of $\cc = \REC$, we deduce the following. 

\begin{theorem}\label{Thm:RECSubgroup-characterization}
Let $G$ be a finitely generated group. Then $G$ is $\REC^\forall$-flat (as a group) if and only if $G$ has decidable subgroup membership problem.
\end{theorem}
\begin{proof}
Let $H \leqfg G$ be as given, and let $\pi \colon A^\ast \to G$ be surjective, with $A$ any finite generating set. First, suppose that $G$ is $\REC^\forall$-flat. Then since $H \in \REC^\forall(H)$, we have by $\REC^\forall$-flatness of $G$ that $H \in \REC^\forall(G \mid H) \subseteq \REC^\forall(G)$. Thus $\pi^{-1}(H)$ is recursive, and hence we can decide membership in $H$. Conversely, suppose that we can decide membership in $H$. Let $K \subseteq H$ be such that $K \in \REC^\forall(H)$. We must show that $K \in \REC^\forall(G \mid H)$, i.e.\ that $K$ is a recursive subset of $G$. But this is easy: $K$ is a recursive subset of $G$ if and only if we can decide membership in $K$ inside $G$. To decide if a given word $w \in A^\ast$ represents an element of $K$, we first simply check whether $w$ represents an element of $H$, which is decidable by assumption. If it does not represent such an element, then it does not represent an element of $K$. If it does represent an element of $H$, then since $H$ is finitely generated we can also find a representative $w'$ for $w$ in terms of the given finite generating set of $H$. Since $K \in \REC^\forall(H)$, we can now decide if $w'$ represents an element $K$, which is equivalent to $w$ representing an element of $K$. 
\end{proof}

The obvious analogue of Proposition~\ref{Prop:Counterexample-machine} for membership in finitely generated submonoids and flatness as a monoid also hold, with an identical proof \textit{mutatis mutandis}. This yields the following analogue of Theorem~\ref{Thm:RECSubmonoid-generalization}.

\begin{theorem}\label{Thm:RECSubmonoid-generalization}
Let $M$ be a finitely generated monoid. Then $M$ is $\REC^\forall$-flat (as a monoid) if and only if $G$ has decidable submonoid membership problem.
\end{theorem}

Equipped with these results, and in particular with Theorem~\ref{Thm:RECSubgroup-characterization}, we are now ready to provide our first counterexamples to Conjecture~\ref{Conj:Main-false-conjecture}, and thereby show that having word problem in $\cc$ and being $\cc^\forall$-flat are in general distinct properties.

\begin{corollary}\label{Cor:REC-is-a-counterexample}
The following three classes are counterexamples to Conjecture~\ref{Conj:Main-false-conjecture}:
\begin{enumerate}
\item $\pCF$, the class of poly-context-free languages;
\item $\CS$, the class of context-sensitive languages;
\item $\REC$, the class of recursive languages.
\end{enumerate}
Indeed, for each of these classes $\cc$, the direct product of two non-abelian free groups $F_2 \times F_2$ has word problem in $\cc$, but it is not $\cc^\forall$-flat.
\end{corollary}
\begin{proof}
It is a classical result due to Mikhailova \cite{Mikhailova1958} that the direct product $F_2 \times F_2$ contains a fixed finitely generated subgroup in which membership is undecidable. It is also known by Brough \cite{Brough2014} that $F_2 \times F_2$ has word problem in $\pCF$, and furthermore we have $\pCF \subset \CS \subset \REC$. Combining these results with Proposition~\ref{Prop:Counterexample-machine} now yields the corollary. 
\end{proof}

This gives some insight into why the class of $\CF^\forall$-groups is classifiable (as in Theorem~\ref{Thm:REG-forall-CF-forall-characterized}): indeed, the structure of groups with context-free word problem is sufficiently restricted by the Muller--Schupp Theorem to show that all subgroups of such groups are themselves context-free, which eliminates any possible application of Proposition~\ref{Prop:Counterexample-machine}. In this line, we are also led to the following natural question. 

\begin{question}\label{Question:Characterize-ETOL-A-flat}
Is every group with word problem in $\ETOL$ also $\ETOL^\forall$-flat?
\end{question}

It remains an open problem whether the class of groups with word problem in $\CF$ is strictly contained in the class of groups with word problem in $\ETOL$, or whether the two classes coincide. Indeed, it is even an open problem whether $\Z^2$ has word problem in $\ETOL$. Nevertheless, it is conceivable that sufficiently strong structural properties can be deduced about subgroups of groups with word problem in $\ETOL$ to give an affirmative answer to Question~\ref{Question:Characterize-ETOL-A-flat}. For example, it can be shown that the free product of two groups with $\ETOL$ word problem again has $\ETOL$ word problem \cite{NybergBrodda2023}. 

\subsection{Recursively enumerable flatness}\label{Subsec:Recursively-enumerable-flatness}

We now turn to considering $\cc = \RE$, the class of recursively enumerable languages. In this case, we can give a full characterization of $\RE^\forall$-flatness. 

\begin{theorem}\label{Thm:Characterizing-REA-groups-Higman}
Let $G$ be a finitely generated group. The following are equivalent: 
\begin{enumerate}
\item $G$ is $\RE^\forall$-flat.
\item $G$ has word problem in $\RE$.
\item $G$ is recursively presented.
\end{enumerate}
\end{theorem}
\begin{proof}
The equivalence of (2) and (3) is obvious; The fact that (1) implies (2) is simply an application of Lemma~\ref {Lem:If-CAflat-then-WP-C}. We will now prove that (2) implies (1). Suppose that $G$ has word problem in $\RE$ and let $H\leqfg G$. To prove that  $G$ is $\RE^\forall$-flat, we only have to show that $\RE^\forall(H)\subseteq \RE^\forall(G\mid H)$. Let $X\in \RE^\forall(H)$. To prove that $X\in \RE^\forall(G)$ is to say that  the set of words over the generators of $G$ which represent elements of $X$ can be recursively enumerated. We can do it as follows:
start enumerating elements of $X$ as words over $H$, which is possible as $X\in \RE^\forall(H)$, and every time we produce an $H$-word representing an element of $X$, we start enumerating in parallel words over the generators of $G$ equal to that word, which is possible since $\WP(G)\in \RE$.
This procedure will eventually enumerate all words over the elements of $G$ equal to some element of $X$. Hence, $X\in \RE^\forall(G)$.
\end{proof}

Recall Higman's embedding theorem, proved in \cite{Higman1961}, which states that a finitely generated group is recursively presented if and only if it is a subgroup of a finitely presented group. This gives an additional point of view of Theorem~\ref{Thm:Characterizing-REA-groups-Higman}, showing that the $\RE^\forall$-flat are precisely the subgroups of finitely presented groups. 

We conclude from Theorem~\ref{Thm:Characterizing-REA-groups-Higman} that the class $\cc = \RE$ is not a counterexample to Conjecture~\ref{Conj:Main-false-conjecture}. On the other hand, as in the case of $\CF^\exists$-flat groups, the question of algebraically characterizing $\RE^\exists$-flat groups seems difficult. We have the following general result as a first step.

\begin{proposition}\label{Prop:REA-is-REE-flat}
Every $\RE^\forall$-flat group is $\RE^\exists$-flat.
\end{proposition}
\begin{proof}
Let $G$ be a  $\RE^\forall$-flat group, $H\leqfg G$ and $L\in \RE^\exists(G\mid H)$. In order to prove that $G$ is  $\RE^\exists$-flat, it suffices to prove that $X\in \RE^\exists(H)$, by Lemma \ref{Lem:Elementary-passing-to-submonoid-2.5+2.6}(2). We will present an algorithm to recursively enumerate a language of words over the generators of $H$ representing all elements in $X$. Since $X\in \RE^\exists(G)$, we run the algorithm for enumerating the elements of $X$ as words over the generators of $G$, and in parallel enumerate all $G$-words equivalent to those words until we find such a word over the generators of $H$, which must certainly exists as $L\subseteq H$.
\end{proof}

The following result, using Tarski monsters, shows that the class of $\RE^\forall$-flat groups is in general strictly contained inside the class of $\RE^\exists$-flat groups.

\begin{proposition}\label{Prop:Tarski-monsters}
There exist $\RE^\exists$-flat groups which are not $\RE^\forall$-flat.
\end{proposition}
\begin{proof}
Ol'shanskii \cite{Olshanskii1979} (proof given in \cite[Theorem~28.7]{Olshanskii1989}) proved that for every sufficiently large prime number $p$, there are continuum many finitely generated and pairwise non-isomorphic groups of exponent $p$ all of whose proper subgroups have order $p$. Let $G$ be any such group. Then clearly $G$ is $\RE^\exists$-flat: if $H \leqfg G$, then either $H = G$ in which case $\RE^\exists(G \mid H) = \RE^\exists(H)$ by definition; or else $H \neq G$ in which case $H$ is finite. Suppose $K \in \RE^\exists(G \mid H)$. Then $K \subseteq H$ is a finite set, and hence there is a finite (and thus recursively enumerable) set of words in any generating set of $H$ representing $K$. Hence $K \in \RE^\exists(H)$. By Lemma~\ref{Lem:Elementary-passing-to-submonoid-2.5+2.6}(2), it follows that $\cc^\exists(G \mid H) = \cc^\exists(H)$. Since $H$ was arbitrary, $G$ is $\RE^\exists$-flat. In particular, there are continuum many $\RE^\exists$-flat groups. However, there are only countably many $\RE^\forall$-flat groups by Theorem~\ref{Thm:Characterizing-REA-groups-Higman}, since there are only countably many recursively presented groups (indeed, there are only countably many Turing machines). This completes the proof. 
\end{proof}

\begin{remark}
We remark that, proceeding in a similar way, we can prove that the groups used in the proof of Proposition \ref{Prop:Tarski-monsters} are in fact $\cc^\exists$-flat for any trio $\cc$. Moreover, since not all such groups have decidable word problem, this proves the existence of $\CF^\exists$-flat groups with undecidable word problem (cf.\ Question \ref{question: characterize cfflat}).
\end{remark} 

This gives rise to the following natural question.

\begin{question}\label{Question: characterize reflat}
Is there an algebraic characterization of the class of $\RE^\exists$-flat groups?
\end{question}

Indeed, the examples given in the proof of Proposition~\ref{Prop:Tarski-monsters} show that there exist $\RE^\exists$-flat groups which do not have word problem in $\RE$. It would be interesting to know any further connections between the classes of epi-$\RE$ groups and $\RE^\exists$-flat groups; note that for recursively presented groups, in particular finitely presented groups, the condition of epi-$\RE$ vacuously implies $\RE^\forall$-flatness, but the converse is not true by Theorem~\ref{Thm:Characterizing-REA-groups-Higman}.

\section{Open problems and questions}\label{Sec:5openproblems}

\noindent In this section, we collect some open problems and directions for further research. First, in \S\ref{Subsec:FinalConjugacyNote} we discuss some tentative connections between $\cc^\bullet$-subsets of a group and conjugacy classes in groups. Finally, in \S\ref{Subsec:open-problems} we collect the questions we have posed throughout the article in one convenient location, and add some further to our list.

\subsection{Connections with conjugacy classes}\label{Subsec:FinalConjugacyNote}

Let $\cc$ be a class of languages closed under inverse homomorphism. For convenience, we will denote the word problem \eqref{Eq:WP-definition} of $G$ with respect to some arbitrarily chosen finite generating set by $\WP(G)$, and as discussed in \S\ref{Sec:1Prelim} this has a definite meaning since $\cc$ is closed under inverse homomorphism. Analogously, we say that a group has conjugacy problem in $\cc$ if all conjugacy classes of elements of $G$ belong to $\cc^\forall(G)$. Since the identity is only conjugate to itself, the set $\{1\}$ is itself a conjugacy class. Hence, it is clear that $\CP(G)\in \cc$ is a stronger condition than $\WP(G)\in \cc$. On the other hand, since there are groups with decidable word problem but undecidable conjugacy problem (first proved by Fridman \cite{Fridman1960}) it follows that at least for the class $\cc = \REC$ of recursive languages, the condition $\CP(G) \in \cc$ is strictly stronger than $\WP(G) \in \cc$. On the other hand, it clearly follows from Theorem~\ref{Thm:AnisimovMSHerbst} that $\WP(G) \in \REG$ if and only if $\CP(G) \in \REG$ (if and only if $G$ is finite), and Levine \cite{Levine2023} has proved that the same equivalence is also true for the class $\CF$ of context-free languages.

In this small section, we will prove that the conditions also coincide for the class of one-counter languages $\cc = \OC$. We note that this is also one of the few classes in which Conjecture~\ref{Conj:Main-false-conjecture} holds (and, as proved in \S\ref{Subsec:Rec-and-CS}, the conjecture fails in general). In this line, we thus state the following proposition, which we will combine with Theorem~\ref{Thm:OC-forall-characteriztion.}. 

\begin{proposition}\label{Prop:OC-ConjugacyProblem}
Let $G$ be a finitely generated group. The following are equivalent: 
\begin{enumerate}
\item $G$ has a one-counter word problem.
\item $G$ has a one-counter conjugacy problem.
\item $G$ is virtually cyclic. 
\end{enumerate}
\end{proposition}
\begin{proof}
The equivalence of (1) and (3), due to Herbst \cite{Herbst1991}, appears in Theorem~\ref{Thm:OC-forall-characteriztion.}, and as discussed above (2) implies (1). Hence, it suffices to prove that (3) implies (2). Suppose that $G$ is virtually cyclic. If $G$ is finite, then, obviously, $\CP(G)\in \OC$, so we assume that $G$ is virtually $\Z$. Let $H$ be a finite index infinite cyclic subgroup of $G$. We may assume that $H$ is normal as every finite index subgroup contains a normal subgroup and every finite index subgroup of $\Z$ is infinite cyclic. Let $b_1,\ldots, b_n$ be a right transversal for $H$ in $G$.
For each $1 \leq i \leq n$, let $\varphi_i\in \Aut(H)$ be defined as $x\mapsto b_ixb_i^{-1}$. It is easy to see that $\varphi_i$ is an automorphism for all $1 \leq i \leq n$, and so it must be the identity or the inverse automorphism.

Let $g\in G$ and $h\in H$, and let $1 \leq i \leq n$ be such that $g=hb_i$. For $h'\in H$ and $1 \leq j \leq n$, we have that
\[
(hb_i)^{h'b_j}=b_j^{-1}h'^{-1} hb_ih'b_j=(h'^{-1}h(h'\varphi_i))\varphi_j^{-1}b_j^{-1}b_ib_j.
\]
If $\varphi_i$ is the identity automorphism, then $h'^{-1}h(h'\varphi_i)=h'^{-1}hh'=h,$ and so 
\[
(hb_i)^{h'b_j}=(h\varphi_j^{-1})b_j^{-1}b_ib_j.
\]
This implies that the conjugacy class of $g$ is finite, and in particular belongs to $\OC^\forall (G)$. On the other hand, if $\varphi_i$ is the inverse automorphism, then we have that $h'^{-1}h(h'\varphi_i)=h'^{-2}h$, so 
\[
(hb_i)^{h'b_j}=(h'^{-2}h)\varphi_j^{-1}b_j^{-1}b_ib_j
\] 
and letting 
\[
K_j=\{(h'^{-2}h)\varphi_j^{-1}b_j^{-1}b_ib_j\mid h'\in H\},
\]
we have that the conjugacy class of $g$ is $\bigcup_{j\in [n]} K_j$. Let now $1 \leq j \leq n$. We will show that $K_j\in \OC^\forall(G)$. Let $2H$ denote the subgroup of squares of elements in $H$. If $\varphi_j$ is the identity, then $K_j=2H(hb_j^{-1}b_ib_j)$, which is the coset of a finite index subgroup of $H$ and consequently also of a finite index subgroup of $G$. In this case, $K_j\in \REG^\forall(G)\subseteq \OC^\forall(G).$ If $\varphi_j$ is instead the inverse automorphism, then we have $K_j=2H(h^{-1}b_j^{-1}b_ib_j)$ and the same argument shows that $K_j\in \REG^\forall(G)$. But $\REG^\forall(G) \subseteq \OC^\forall(G)$, and so we are done. 
\end{proof}

Combining the above with Theorem~\ref{Thm:OC-forall-characteriztion.}, we obtain the following. 

\begin{corollary}\label{Cor:CP-is-OC-flatness}
Let $G$ be a finitely generated group. Then $G$ has one-counter conjugacy problem if and only if $G$ is $\OC^\forall$-flat. 
\end{corollary}

It would be interesting to understand for which other classes of languages we can find a connection between the conjugacy problem and $\cc^\forall$-flatness. In particular, the class $\OC$ is sufficiently restricted to make discussing the word problem and the conjugacy problem the same question; in this line, we ask the following question, which may shed some light in separating these problems and $\cc^\forall$-flatness.

\begin{question}\label{question: wp-cp-flat}
Is there a class of languages $\cc$ such that there exists a group $G$ with:
\begin{enumerate}
\item $\WP(G) \in \cc$,
\item $\CP(G) \notin \cc$, and
\item $G$ is $\cc^\forall$-flat. 
\end{enumerate}
\end{question}

An affirmative answer to this question would a new way to separate groups by their language-theoretic properties, not covered by the traditional hierarchy of their word problem. In general, a better understanding of connections between the conjugacy problem and language-theoretic properties of subsets of groups seems like a promising line of research.

\subsection{List of open problems}\label{Subsec:open-problems}

We first collect, for convenience, the open problems and questions that we have stated throughout this article:

\newtheorem*{question1*}{Question \ref{Quest:can-square-collapse}}
\newtheorem*{question2*}{Question  \ref{question: characterize cfflat}}
\newtheorem*{question3*}{Question \ref{Quest:Is-ccE-flatness-closed-wrt-finite-extensions}}
\newtheorem*{question4*}{Question  \ref{Question:Characterize-ETOL-A-flat}}
\newtheorem*{question5*}{Question \ref{Question: characterize reflat}}
\newtheorem*{question6*}{Question \ref{question: wp-cp-flat}}

\begin{question1*}
Do there exist classes of languages $\cc_1, \cc_2$ such that $\cc_1 \subsetneq \cc_2$ but $\cc_1^\exists(G) = \cc_2^\forall(G)$ for all finitely generated groups (or monoids) $G$? \end{question1*}

\begin{question2*}
Is there an algebraic characterization of $\CF^\exists$-flat groups? In particular, is the word problem decidable in every $\CF^\exists$-flat group? 
\end{question2*}


\begin{question4*}
Is every group with word problem in $\ETOL$ also $\ETOL^\forall$-flat?
\end{question4*}

\begin{question5*}
Is there an algebraic characterization of the class of $\RE^\exists$-flat groups?
\end{question5*}

\begin{question6*}
Is there a class of languages $\cc$ such that there exists a group $G$ with:
\begin{enumerate}
\item $\WP(G) \in \cc$,
\item $\CP(G) \notin \cc$, and
\item $G$ is $\cc^\forall$-flat. 
\end{enumerate}
\end{question6*}

Having stated the questions that have appeared throughout the article, we now proceed by stating several further questions that appear naturally in the broader context of this article. First, recall from \S\ref{Sec:4flatgroupsforspecific} that we have proved for $\cc\in\{\REG,\OC,\CF,\RE\}$ that $\cc^\forall$-flatness implies $\cc^\exists$-flatness. Unlike Conjecture~\ref{Conj:Main-false-conjecture}, we do not have any counterexamples to such a statement holding in general, and thus we wonder whether it is always the case:

\begin{question}
If $\cc$ is a (reasonable) class of languages, does $\cc^\forall$-flatness imply $\cc^\exists$-flatness? Is this true for $\cc\in\{\CS,\REC\}$?
\end{question}

Note here that a \textit{reasonable} class of languages may be interpreted as one which poses no immediate obstructions to the question; it is possible that even a full semi-$\AFL$
 may suffice for \textit{reasonable}, for example. 

We would also like to understand the precise relationship between $\cc^\bullet_1$-flatness and $\cc^\bullet_2$-flatness when $\cc_1, \cc_2$ are two classes of languages such that one is included in the other. Thus far, none of our investigations allows us to give a negative answer to any of the following questions:

\begin{question}
Let $\cc_1,\cc_2$ be (reasonable) classes of languages, such that $\cc_1\subseteq \cc_2$. 
\begin{enumerate}
\item Does  $\cc_1^\forall$-flatness imply $\cc_2^\forall$-flatness?
\item Does $\cc_2^\exists$-flatness imply $\cc_1^\exists$-flatness?
\end{enumerate}
Can we answer any of the two questions if $\cc_1\in\{\REG, \CF\}$?
\end{question}

Finally, recall that Ho \cite{ho2018} has proved that the word problem of the free abelian group $\Z^m$ is a multiple context-free language ($\MCF$), following a previous result by Salvati \cite{salvati2015}. In general, not much is known about the class of groups having a multiple context-free word problem. Since the class of multiple context-free languages is a full semi-AFL \cite{sekietal}, the following question is then natural.

\begin{question}
Is $\Z^m$ a $\MCF^\forall$-flat group? 
\end{question}

Thus far, we have no indication whether the above question ought to have a positive or negative answer. Nevertheless, it is a very concrete question, and any answer to it would likely be fruitful in attempting to understand both flatness and multiple context-free groups more generally.

\bibliographystyle{amsalpha}
\bibliography{LinguisticSubsets-v2-03-2025.bib}

 \end{document}